\documentclass{siamltex}
\usepackage{amsfonts,amsmath,amssymb}
\usepackage{mathrsfs,mathtools}
\usepackage{enumerate}
\usepackage{xspace}
\usepackage{hyperref}
\usepackage{esint}
\usepackage{graphicx}
\DeclareMathAlphabet{\mathpzc}{OT1}{pzc}{m}{it}

\usepackage{color}

\newcommand{\calC}{{\mathcal C}}
\newcommand{\calD}{{\mathcal D}}
\newcommand{\calE}{{\mathcal E}}

\newcommand{\calH}{{\mathcal H}}
\newcommand{\calI}{{\mathcal I}}

\newcommand{\calK}{{\mathcal K}}
\newcommand{\calL}{{\mathcal L}}

\newcommand{\calP}{{\mathcal P}}

\newcommand{\calR}{{\mathcal R}}

\newcommand{\calU}{{\mathcal U}}

\newcommand{\calW}{{\mathcal W}}

\newcommand{\calZ}{{\mathcal Z}}

\newcommand{\polH}{{\mathbb H}}

\newcommand{\polN}{{\mathbb N}}

\newcommand{\polP}{{\mathbb P}}
\newcommand{\polQ}{{\mathbb Q}}

\newcommand{\polV}{{\mathbb V}}

\newcommand{\frakD}{{\mathfrak D}}
\newcommand{\frakE}{{\mathfrak E}}

\newcommand{\frakd}{{\mathfrak d}}

\newcommand{\Y}{{\mathpzc{Y}}}
\newcommand{\ve}{{\mathscr{U}}}
\newcommand{\C}{{\mathcal{C}}}
\newcommand{\K}{{\mathpzc{K}}}

\newcommand{\Ind}{\textup{\textsf{Ind}}}

\newcommand{\DIV}{\nabla\!{\cdot}}   
\newcommand{\GRAD}{\nabla}           
\newcommand{\LAP}{{\Delta}}          

\def\Ldeux{{{  L}^2   (\Omega)}}

\def\Hunz{{{   H}^1_0 (\Omega)}}

\def\Hs{{{   \mathbb{H}}^s   (\Omega)}}
\def\Hsd{{{   \mathbb{H}}^{-\!s\!}   (\Omega)}}

\def\Hsp1d{{{   \bf H}^{s+1}   (\Omega)}}

\def\Hdeux{{{   H}^2   (\Omega)}}

\newcommand{\Real}{\mathbb R}

\newcommand{\diff}{\, \mbox{\rm d}}
\newcommand{\vare}{{\varepsilon}}
\newcommand{\dt}{{\tau}}

\DeclareMathOperator*{\tr}{tr_\Omega}
\DeclareMathOperator*{\diam}{diam}
\DeclareMathOperator{\Var}{Var}

\newcommand{\ie}{i.e.,\@\xspace}

\newcommand{\cf}{cf.\@\xspace}

\newcommand{\mae}{a.e.~}

\newcommand{\ue}{\textup{\textsf{u}}}
\newcommand{\fe}{\textup{\textsf{f}}}

\def\scl{\left\langle}
\def\scr{\right\rangle}

\newcommand{\logLip}{{\textup{\textsf{logLip}}(\Omega)}}

\newcommand{\GL}{{\textup{\textsf{GL}}}}

\newcommand{\Laps}{{(-\Delta)^s}}
\newcommand{\ds}{\textup{\textsf{d}}_s}

\newcommand{\HL}{ \mbox{ \raisebox{6.9pt} {\tiny$\circ$} \kern-10.3pt} {H_L^1} }
\newcommand{\T}{\mathscr{T}}
\newcommand{\HLn}{\mbox{\,\raisebox{4.7pt} {\tiny$\circ$} \kern-9.3pt}{H_L^1} }  
\newcommand{\vero}{\texttt{v}}
\newcommand{\N}{\mathcal{N}}
\newcommand{\R}{\mathbb{R}}

\newtheorem{remark}[theorem]{Remark}
\numberwithin{equation}{section}

\title{Finite element approximation of the parabolic fractional obstacle problem\thanks{EO has been supported in part by NSF grant DMS-1411808 and by CONICYT through project Anillo ACT1106. AJS is partially supported by NSF grant DMS-1418784.}}

\author{Enrique Ot\'arola\thanks{Departamento de Matem\'atica, Universidad T\'ecnica Federico Santa Mar\'ia, Valpara\'iso, Chile. \texttt{enriqueotarola@gmail.com}.}
\and
Abner J.~Salgado\thanks{Department of Mathematics, University of Tennessee, Knoxville, TN 37996, USA.
\texttt{asalgad1@utk.edu}}}

\begin{document}

\maketitle

\begin{abstract}
We study a discretization technique for the parabolic fractional obstacle problem in bounded domains. The fractional Laplacian is realized as the Dirichlet-to-Neumann map for a nonuniformly elliptic equation posed on a semi-infinite cylinder, which recasts our problem as a quasi-stationary elliptic variational inequality with a dynamic boundary condition. The rapid decay of the solution suggests a truncation that is suitable for numerical approximation. We discretize the truncation with a backward Euler scheme in time and, for space, we use first-degree tensor product finite elements. We present an error analysis based on different smoothness assumptions.
\end{abstract}


\begin{AMS}{
35J70,                    
35R11,                    
35R35,                    
49M15,                    
49M20,                    
49M25,                    
65N12,                    
65N30,                    
65N50}
\end{AMS}

\section{Introduction}
\label{sec:intro}
In this work we shall be interested in the design and analysis of a finite element approximation of the so-called \emph{parabolic fractional obstacle problem}. Let $\Omega$ be an open and bounded subset of $\Real^d$ with $d\geq 1$. Given $s\in (0,1)$, an obstacle $\psi:\Omega \to \Real$, an initial datum $\ue_0:\Omega \to \Real$ and a forcing term $\fe:\Omega\times(0,T] \to \Real$, the parabolic fractional obstacle problem asks for a function $\ue :\Omega \times [0,T] \to \Real$ that satisfies the evolution variational inequality:
\begin{equation}
\label{eq:obstrong}
  \min \left\{ \diff_t \ue + \Laps \ue - \fe, \ue - \psi \right\} = 0
\end{equation}
and $\ue_{|t=0} = \ue_0$. Here, $\Laps$ denotes the fractional powers of the Laplace operator, supplemented with homogeneous Dirichlet boundary conditions, which for convenience we will simply call the fractional Laplacian. We refer to \S\ref{sub:fracLaplace} for a precise definition. 
We must immediately remark that although our exposition is for the fractional Laplacian, our techniques and results are equally applicable to fractional powers of a symmetric and uniformly elliptic second order differential operator $L$, supplemented with homogeneous Dirichlet boundary conditions:
$
L w = - \DIV( A \GRAD w ) + c w
$,
with $A \in C^{0,1}(\bar\Omega,\GL(\Real^d))$ symmetric and positive definite and $0 \leq c \in C^{0,1}(\bar\Omega,\Real)$. The only caveat is that, at the time of this writing, no regularity results are available for \eqref{eq:obstrong} with the fractional Laplacian replaced by $L^s$. 

The study of numerical techniques for nonlocal problems is a rapidly growing field of research. Fractional diffusion has received a great deal of attention in diverse areas of science and engineering such as mechanics \cite{atanackovic2014fractional}, biophysics \cite{bio}, turbulence \cite{wow}, image processing \cite{GH:14}, peridynamics \cite{HB:10} and nonlocal electrostatics \cite{ICH}. In particular, the study of constrained minimization problems such as the parabolic fractional obstacle problem \eqref{eq:obstrong} has received considerable attention. This type of problems arises, for instance, in financial mathematics as a pricing model for American options. The function $\ue$ represents the rational price of a non perpetual American option where the assets' prices are modeled by a L\'evy process, and the payoff function is $\psi$; see \cite{MR2064019,NTZ:10,MR1424787}.

Based on the Caffarelli-Silvestre extension \cite{CS:07} in previous work we provided a comprehensive analysis of the discretization of the linear elliptic case \cite{NOS}, evolution equations with fractional diffusion and Caputo fractional time derivative \cite{NOS3} and the elliptic fractional obstacle problem \cite{NOS4}. In this work we proceed in our research program and show the flexibility of the ideas developed in \cite{NOS} by studying the parabolic fractional obstacle problem \eqref{eq:obstrong}. To the best of our knowledge, this is the first work that addresses the numerical approximation of this problem.

Our presentation is organized as follows. The notation and functional setting is described in Section~\ref{sec:notation}, where we also briefly describe, in \S\ref{sub:fracLaplace}, the definition of the fractional Laplacian, its localization via the Caffarelli-Silvestre extension (\S\ref{sub:CafarrelliSilvestre}) and the well-posedness of the fractional parabolic obstacle problem; see \S\ref{sub:fracobstacle}. The numerical analysis of problem \eqref{eq:obstrong} begins in Section~\ref{sec:truncation}, where we discuss a domain truncation that allows us, in subsequent sections, to consider a space discretization using first-degree tensor product finite elements. The time discretization and its error analysis is described in Section~\ref{sec:timediscr}. The space discretization and its analysis is the content of Section~\ref{sec:space}: we provide an error analysis with minimal (\S\ref{sub:minreg}) and maximal (\S\ref{sub:maxreg}) regularities. This analysis relies on the construction and approximation properties of a positivity preserving interpolant. For $s>3/8$ we construct an interpolant with the requisite properties in Section~\ref{sub:interpolant}.

\section{Notation and preliminaries}
\label{sec:notation}

In this work $\Omega$ is a convex bounded and open subset of $\Real^d$ ($d\geq1$) with polyhedral boundary. Our ideas are equally applicable to domains with curved boundaries, but the exposition becomes rather cumbersome and so we prefer to avoid it. We will follow the notation of \cite{NOS} and define the semi-infinite cylinder and its lateral boundary by
$ \C = \Omega \times (0,\infty)$, $\partial_L\C = \partial\Omega \times (0,\infty)$.
For $\Y>0$ we define the truncated cylinder $\C_\Y = \Omega \times (0, \Y)$ and its lateral boundary $\partial_L \C_\Y = \partial\Omega \times (0,\Y)$. We also define the Dirichlet boundary $\Gamma_D = \partial_L \C_\Y \cup \Omega \times \{\Y\}$. Since we will be dealing with objects defined on $\Real^d$ and $\Real^{d+1}$, it will be convenient to distinguish the $d+1$-dimension. For $x \in \Real^{d+1}$, we denote
\[
  x = (x^1,\cdots,x^{d},x^{d+1}) = (x',x^{d+1}) = (x',y), \qquad x' \in \Real^d, y \in \Real.
\]

Whenever $X$ is a normed space we denote by $\|\cdot\|_X$ its norm and by $X'$ its dual. For normed spaces $X$ and $Y$ we write $X \hookrightarrow Y$ to indicate continuous embedding. We will follow standard notation for function spaces \cite{Adams,Tartar}. In addition, for an open set $D \subset \Real^N$, $N \geq 1$, if $\omega$ is a weight and $p \in (1,\infty)$ we denote the Lebesgue space of $p$-integrable functions with respect to the measure $\omega \diff x$ by $L^p(\omega,D)$; see \cite{HKM,Kufner80,Turesson}. Similar notation will be used for weighted Sobolev spaces. If $T>0$ and $\phi: D \times [0,T] \to \Real$, we consider $\phi$ as a function of $t$ with values in a Banach space $X$, $\phi :[0,T] \ni t \mapsto \phi(t) \equiv \phi(\cdot,t) \in X$. For $1 \leq p \leq \infty$ we will say that $\phi \in L^p(0,T;X)$ if the mapping $t \mapsto \| \phi(t) \|_X$ is in $L^p(0,T)$. We also introduce the space $BV(0,T;X)$ of $X$-valued functions of bounded variation \cite[Definition A.2]{Brezis}
\[
 \Var_X g:= \sup_{\calP} \left\{ \sum_{j=1}^{J} \|g(r_j) - g(r_{j-1}) \|_X \right\} < \infty,
\]
where the supremum is taken over all partitions 
$
  \calP=\{0=r_0 <  \ldots < r_j < \ldots < r_J = T\}
$
of the time interval $[0,T]$. We recall that if $g \in BV(0,T;X)$, then at every point $t_0 \in [0,T)$ there exists the \emph{right} limit $g_{+}(t_0) = \lim_{t \downarrow t_0} g(t)$ \cite[Lemma A.1]{Brezis}.

Let $\calK \in \polN$ be the number of time steps. We define the uniform time step as $\dt = T/\calK$ and we set $t_k = k \dt$, $k=0,\ldots, \calK$. Given a function $w :[0,T] \to X$, we denote $w^k = w(t_k) \in X$ and $w^\dt = \{ w^k \}_{k=0}^\calK \subset X$. For any sequence $W^\dt \subset X$, we define the piecewise constant interpolant $\bar{W}^\dt \in L^\infty(0,T;X)$ by
\begin{equation*}
\label{eq:defofbar}
  \bar{W}^\dt(t) = W^{k+1} \quad t \in (t_k,t_{k+1}], \quad k = 0, \ldots, \calK-1.
\end{equation*}
We also define the piecewise linear interpolant $\hat{W}^\dt \in C([0,T];X)$ by
\begin{equation*}
\label{eq:defofhat}
   \hat{W}^\dt(t) = \frac{t - t_k}\dt W^{k+1} + \frac{t_{k+1} - t}\dt W^k \quad t \in [t_k,t_{k+1}], \quad k = 0, \ldots, \calK-1.
\end{equation*}
The first order backward difference operator $\frakd$ is defined by $ \frakd W^{k+1} = W^{k+1} - W^k $.
We note that $\diff_t \hat{W}^\dt(t) = \dt^{-1} \frakd W^{k+1}$ for all $t \in (t_k,t_{k+1})$ and $k = 0, \ldots, \calK-1$. Finally, we also notice that, for any sequence $W^\dt \subset X$ and $p \in [1,\infty)$ we have
\[
  \left\| \bar{W}^\dt \right\|_{L^p(0,T;X)} = \left( \dt \sum_{k=1}^{\calK} \| W^k\|_X^p \right)^{1/p},
\]
and
$
  \| \bar{W}^\dt \|_{L^\infty(0,T;X)} = \| \hat{W}^\dt \|_{L^\infty(0,T;X)} 
  = \max \left\{ \left\|W^k \right\|_X : k=0,\ldots, \calK \right\}
$.

The relation $a\lesssim b $ indicates that $a \leq C b$ for a constant that does not depend on either $a$ or $b$, but it might depend on the problem data. The value of $C$ might change at each occurrence. 

\subsection{The fractional Laplacian}
\label{sub:fracLaplace}

For a bounded domain there are several ways, not necessarily equivalent, to define the fractional Laplacian; see \cite{NOS} for a discussion. As in \cite{NOS} we will adopt that based on spectral theory \cite{BS}. Namely, since $-\LAP: \calD(\LAP) \subset \Ldeux \to \Ldeux$ is an unbounded, positive and closed operator with dense domain $\calD(\LAP) = \Hunz \cap \Hdeux$ and its inverse is compact, there is a countable collection of eigenpairs $\{\lambda_l, \varphi_l \}_{l \in \polN} \subset \Real^+ \times \Hunz$ such that $\{\varphi_l\}_{l \in \polN}$ is an orthonormal basis of $\Ldeux$ and an orthogonal basis of $\Hunz$. If $w \in C_0^\infty(\Omega)$
\[
  w = \sum_{l \in \polN} w_l \varphi_l, \qquad w_l = \int_\Omega w \varphi_l \diff x',
\]
then, for any $s \in (0,1)$, we define
$  \Laps w = \sum_{l \in \polN} \lambda_l^s w_l \varphi_l$, 
As it is well known, the theory of Hilbert scales presented in \cite[Chapter 1]{Lions} shows that $\calD \left((-\Delta)^{s/2} \right) = \Hs = [\Ldeux,\Hunz]_s$, \ie the real interpolation between $\Ldeux$ and $\Hunz$. Consequently, the definition of $\Laps$ can be extended by density to the space $\Hs$. If, for $0<s<1$, we denote by $\Hsd$ the dual space of $\Hs$, then $\Laps : \Hs \to \Hsd$ is an isomorphism.

\subsection{The Caffarelli-Silvestre extension problem}
\label{sub:CafarrelliSilvestre}

The Caffarelli-Silvestre result \cite{CS:07,NOS}, 
requires us to deal with a nonuniformly elliptic equation. With this in mind, we define the weighted Sobolev space
\begin{equation*}
\label{eq:defofHL}
  \HL(y^{\alpha},\C) = \left\{ w \in H^1(y^\alpha,\C): w = 0 \textrm{ on } \partial_L \C\right\}.
\end{equation*}
Since $\alpha \in (-1,1)$, $|y|^\alpha$ belongs to the Muckenhoupt class $A_2(\Real^{n+1})$; see \cite{Javier,Turesson}. Consequently, $\HL(y^{\alpha},\C)$ is a Hilbert space, and smooth functions are dense in $\HL(y^{\alpha},\C)$ (\cf \cite[Proposition 2.1.2, Corollary 2.1.6]{Turesson}).

As \cite[(2.21)]{NOS} shows, the following \emph{weighted Poincar\'e inequality} holds:
\begin{equation}
\label{Poincare_ineq}
  \| w \|_{L^2(y^{\alpha},\C)} \lesssim \| \nabla v \|_{L^2(y^{\alpha},\C)}, \quad \forall w \in \HL(y^{\alpha},\C).
\end{equation}
Then, the seminorm of $H^1(y^\alpha,\C)$ is equivalent to the norm in $\HL(y^{\alpha},\C)$. For $w \in H^1(y^{\alpha},\C)$ we denote by $\tr w$  the trace of $w$ onto $\Omega \times \{ 0 \}$. We recall (\cite[Prop.~2.5]{NOS})
\begin{equation}
\label{Trace_estimate}
  \tr \HL(y^{\alpha},\C) = \Hs, \qquad \| \tr w \|_{\Hs} \leq C_{\tr} \| w \|_{ \HLn(y^{\alpha},\C) }.
\end{equation}

The seminal work of Caffarelli and Silvestre \cite{CS:07,NOS} 
showed that the operator $\Laps$ can be realized as the Dirichlet-to-Neumann map for a nonuniformly elliptic boundary value problem. Namely, if $\calU \in \HL(y^{\alpha},\C)$ solves
\begin{equation}
\label{eq:extension}
    -\DIV\left( y^\alpha \GRAD \calU \right) = 0  \text{ in } \C, \qquad
    \calU = 0  \text{ on } \partial_L \C, \qquad
    \partial_\nu^\alpha \calU = \ds f  \text{ on } \Omega \times \{0\},
\end{equation}
where $\alpha = 1-2s$, $\partial_\nu^\alpha \calU = -\lim_{y\downarrow 0} y^\alpha \calU_y$ and $\ds = 2^\alpha \Gamma(1-s)/\Gamma(s)$ is a normalization constant, then $u = \tr\calU 	 \in \Hs$ solves
\begin{equation}
\label{eq:fraclap}
  \Laps u = f.
\end{equation}

The reader is referred to \cite{NOS} for a detailed and thorough exposition on how the groundbreaking identity given by \eqref{eq:extension} can be used to design and analyze an efficient finite element approximation of solutions to \eqref{eq:fraclap}.

\subsection{The parabolic fractional obstacle problem}
\label{sub:fracobstacle}

Given an obstacle $\psi$ that satisfies $\psi \in \Hs \cap C(\bar\Omega)$ and $\psi \leq 0$ on $\partial\Omega$, let 
\begin{equation*}
\label{eq:deofK}
  \K(\Omega) = \left\{ w \in \Hs: w(x') \geq \psi(x')\  \mae x' \in \Omega \right\}
\end{equation*}
be the convex set of admissible functions, and let $\Ind_{\K(\Omega)}$ be its indicator function, which, since $\K(\Omega)$ is closed and convex, is a non-smooth lower-semi-continuous convex function. We consider the energy
\[
  J(\phi) = \frac12 \| \phi \|_{\Hs}^2 + \Ind_{\K(\Omega)}(\phi).
\]
With this notation, the parabolic fractional obstacle problem can be understood as the gradient flow for $J$, or an evolution equation for a maximal monotone operator: Given an initial datum $\ue_0 \in \Ldeux$ and a function $\fe \in L^2(0,T;L^2(\Omega))$, find $\ue$ such that $\ue(0) = \ue_0$ and it solves the differential inclusion
\begin{equation}
\label{eq:gradflow}
  \diff_t \ue(t) + \partial J(\ue(t)) \ni \fe(t) \quad \mae t\in (0,T).
\end{equation}

Problem \eqref{eq:gradflow} can be equivalently understood as an evolution variational inequality: Find $\ue$ such that $\ue(t) \in \K(\Omega)$ and for \mae $t \in (0,T)$ and all $\phi \in \K(\Omega)$
\begin{equation}
\label{eq:varineq}
  ( \diff_t \ue(t), \ue(t) - \phi )_{\Ldeux} + \scl \Laps \ue(t), \ue(t) - \phi \scr \leq ( \fe(t), \ue(t) - \phi )_{\Ldeux}, 
\end{equation}
and $\ue(0) = \ue_0$. Here and in what follows $(\cdot,\cdot)_{\Ldeux}$ denotes the inner product of $\Ldeux$ and $\scl\cdot,\cdot\scr$ the duality pairing between $\Hsd$ and $\Hs$.

From these formulations existence and uniqueness of solutions and a priori estimates can be easily obtained with standard techniques on maximal monotone operators \cite{Brezis}. For instance, if $\ue$ solves \eqref{eq:varineq}, then it satisfies the energy estimate
\begin{equation*}
\label{eq:stability0}
\|\ue \|_{L^\infty(0,T;\Ldeux)}^2 + \| \ue \|^2_{L^2(0,T;\Hs)} \lesssim \frakD^2,
\end{equation*}
where we denoted
\begin{equation}
\label{eq:defoffrakD}
  \frakD^2 = \frakD^2(\ue_0, \fe, \psi ) = \|\ue_0\|_{L^2(\Omega)}^2 + \| \fe \|_{L^2(0,T;L^2(\Omega))}^2 + \| \psi \|_{\Hs}^2.
\end{equation}
If $\fe \in L^2(0,T;L^2(\Omega))$ and $\ue_0 \in \K(\Omega)$, then there exists a unique \emph{strong solution}, that is $\ue \in C^{0}([0,T];L^2(\Omega))$ which is locally absolutely continuous in $(0,T)$ and satisfies \eqref{eq:varineq} at almost every point. In addition, we have that $\ue \in H^1(0,T; L^2(\Omega))$ and the mapping $t \mapsto J(\ue(t))$ is locally absolutely continuous in $(0,T]$, which implies that $\ue \in W_{loc}^{1,1}(0,T;\Hs)$, and that the following estimate holds \cite[Theorem 3.6]{Brezis}
\begin{equation*}
\label{eq:stability}
\|\diff_t \ue (t)\|_{L^2(\Omega)}^2 + \diff_t \| \ue(t)\|^2_{\Hs} = ( \fe, \diff_t \ue(t) )_\Ldeux \quad \mae t \in (0,T).
\end{equation*}
If, moreover, $\ue_0 \in \calD(\partial J)$ and $\fe \in BV(0,T;\Ldeux)$, then $\ue \in C^{0,1}([0,T];L^2(\Omega))$  \cite[Proposition 3.3]{Brezis}. Finally, if $\fe \in W^{1,1}(0,T;\Ldeux)$ we have
\[
 \frac12 \diff_t \|\diff_t \ue(t)\|^2_{L^2(\Omega)} \leq ( \fe_t, \diff_t \ue(t) )_\Ldeux, 
\]
in the distributional sense.

Let us now use the Caffarelli--Silvestre extension detailed in \S\ref{sub:fracLaplace} to write an obstacle problem that is equivalent to \eqref{eq:varineq}. To do this, we define the set
\begin{equation*}
\label{eq:extK}
  \K(\C) = \left\{ w \in \HL(y^\alpha,\C): \tr w(x') \geq \psi(x')\  \mae x'\in \Omega \right\}.
\end{equation*}
Problem \eqref{eq:varineq} can then be equivalently stated as: Find $\ve :[0,T] \to \K(\C)$ such that for \mae $t \in (0,T)$ and every $\phi \in \K(\C)$
\begin{multline}
  ( \tr \diff_t \ve(t), \tr (\ve(t)- \phi) )_{\Ldeux} + a(\ve(t), \ve(t)- \phi ) 
  \\
  \leq ( \fe , \tr( \ve(t) - \phi ) )_{\Ldeux}, 
  \label{eq:gradflowextension}
\end{multline}
with $\tr \ve(0) = \ue_0$. Here the bilinear form $a$ is defined by
\begin{equation*}
\label{eq:defofa}
  a(w,\phi) = \frac1\ds \int_\C y^\alpha \GRAD w \GRAD \phi \diff x' \diff y, \quad \forall w, \phi \in \HL(y^\alpha,\C).
\end{equation*}
The description of the functional setting for problem \eqref{eq:gradflowextension}, together with existence and uniqueness results, follow the analysis developed for problem \eqref{eq:gradflow}; see \cite{Brezis,NSV}. In particular, we have the energy inequality
\begin{equation}
\label{eq:stab_ext}
  \|\tr \ve\|_{L^\infty(0,T;\Ldeux)}^2 + \| \ve \|^2_{L^2(0,T;\HLn(y^{\alpha},\C))} \lesssim \frakD^2,
\end{equation}
where $\frakD$ is defined in \eqref{eq:defoffrakD}. We shall be more specific on the smoothness of the data and the consequences on the regularity of the solution when we perform the discretization and its analysis. Let us now contempt ourselves with mentioning that, provided $\ve$ is sufficiently smooth, the following complementarity system holds:
\begin{equation*}
\label{eq:obsexstrong}
  \calZ := \partial^\alpha_\nu\ve + \ds \tr \diff_t \ve - \ds \fe \ge 0, \quad
  \tr\ve - \psi \ge 0, \quad
  \calZ \left(\tr \ve - \psi \right) = 0.
\end{equation*}

\section{Truncation}
\label{sec:truncation}

The variational inequality \eqref{eq:gradflowextension} is posed on a infinite domain and, consequently, it cannot be directly approximated with finite element-like techniques. A first step towards the discretization is to truncate the domain $\C$ to a bounded cylinder $\C_\Y = \Omega \times (0,\Y)$ and study the effect of this truncation. We begin with a result that shows the exponential decay of the solution to \eqref{eq:gradflowextension}; compare with \cite[Proposition 3.1]{NOS}, \cite[Lemma 4.8]{NOS4} and \cite[Proposition 4.1]{NOS3}.

\begin{lemma}[exponential decay]
\label{lem:expdecay}
If $\ue_0 \in \calK(\Omega)$, $\psi \in \Hs\cap C(\bar\Omega)$ and $\fe \in L^2(0,T;L^2(\Omega))$, then, for every $\Y\geq1$, we have
\[
  \| \GRAD \ve \|_{L^2(0,T;L^2(y^\alpha,\Omega\times(\Y,\infty))} \lesssim e^{-\Y/2} \frakD,
\]
where the hidden constant does not depend on neither $\ve$ nor the problem data.
\end{lemma}
\begin{proof}
Consider, for \mae $t \in (0,T]$, the function $w(t) \in \HL(y^\alpha, \C)$ that solves
\begin{equation}
\label{eq:decay_aux}
  \DIV(y^\alpha \GRAD w(t) ) = 0 \ \text{in } \C, \quad 
  w(t)_{|\partial_L\C} = 0, \quad 
  \tr w(t) = \tr \ve(t) \  \text{on } \Omega \times\{0\}.
\end{equation}
Since $\ve$ solves the fractional parabolic obstacle problem \eqref{eq:gradflowextension} and problem \eqref{eq:decay_aux} has a unique solution, we immediately conclude that for \mae $t \in (0,T]$, we have $w(t) = \ve(t)$.  We now apply the decay estimate of \cite[Proposition 4.1]{NOS} to problem \eqref{eq:decay_aux} to obtain 
\[
  \| \GRAD w(t) \|_{L^2(y^\alpha, \Omega\times(\Y,\infty) )} \lesssim e^{-\Y/2} \|\tr \ve(t)\|_{\Hs}.
\]
Finally, integrating over time, and invoking the trace estimate \eqref{Trace_estimate} and the stability estimate \eqref{eq:stab_ext} for problem \eqref{eq:gradflowextension} in terms of $\ue_0$, $f$ and $\psi$, we arrive at
\[
\| \GRAD w \|_{L^2(0,T;L^2(y^\alpha,\Omega\times(\Y,\infty))} \lesssim e^{-\Y/2} \| \ve \|_{L^2(0,T;\HLn(y^{\alpha},\C_{\Y}))} \lesssim e^{-\Y/2} \frakD,
\]
where $\frakD$ is defined in \eqref{eq:defoffrakD}. This concludes the proof.
\end{proof}

The exponential decay of Lemma~\ref{lem:expdecay} allows us to consider a truncated version of the variational inequality \eqref{eq:gradflowextension}. To write this problem we define, for $\Y\geq1$, the Sobolev space
\begin{equation*}
\label{eq:defofspace}
  \HL(y^\alpha,\C_\Y) = \left\{ w \in H^1(y^\alpha,\C_\Y): w_{|\Gamma_D} = 0  \right\},
\end{equation*}
the convex set of admissible functions
\begin{equation*}
\label{eq:defofKY}
  \K(\C_\Y) = \left\{ w \in \HL(y^\alpha,\C_\Y) : \tr w(x') \geq \psi(x') \ \mae  x' \in \Omega \right\}
\end{equation*}
and the bilinear form
\begin{equation*}
\label{eq:defofaY}
  a_\Y(w,\phi) = \frac1\ds \int_{\C_\Y} y^\alpha \GRAD w \GRAD \phi \diff x' \diff y \quad \forall w,\phi \in \HL(y^\alpha,\C_\Y).
\end{equation*}
With these definitions we consider the following truncated problem: Find $v:[0,T] \to \K(\C_\Y)$ such that $\tr v (0) = \ue_0$ and for \mae $t \in (0,T)$ and all $\phi \in \K(\C_\Y)$
\begin{equation}
\label{eq:truncgradflow}
  ( \tr \diff_t v(t), \tr (v(t) - \phi) )_{\Ldeux} + a_\Y(v(t), v(t)-\phi) \leq ( \fe, \tr (v(t) - \phi) )_{\Ldeux}.
\end{equation}
The analysis of this problem follows that of \eqref{eq:gradflow}, developed in \S\ref{sub:fracobstacle}. For brevity, we only present the energy estimate
\begin{equation}
\label{eq:stab_ext_trun}
  \|\tr v \|_{L^\infty(0,T; \Ldeux)}^2 + \| v\|^2_{L^2(0,T;\HLn(y^{\alpha},\C_\Y))} \lesssim \frakD^2.
\end{equation}

We define $\calH_{\alpha} :\Hs \rightarrow \HL(y^{\alpha},\C_{\Y})$, the truncated $\alpha$-harmonic extension operator as follows: if $w \in \Hs$, then $\calW = \calH_\alpha w \in \HL(y^{\alpha},\C_{\Y})$ solves
\begin{equation}
\label{eq:mathcalH}
\DIV(y^{\alpha} \nabla \mathcal{W}) = 0 \text{ in } \C_\Y, 
\quad
\mathcal{W} = 0 \text{ on } \partial_L \C_\Y \cup \Omega \times \{\Y\} , 
\quad
\mathcal{W} = w\text{ on } \Omega \times \{0\}.
\end{equation}
We recall \cite[Theorem 2.7]{NOS} for problem \eqref{eq:mathcalH}. If $w \in \mathbb{H}^{1+s}(\Omega)$, then
\begin{equation}
\label{eq:reg_estimates}
 \| \nabla \nabla_{x'} \mathcal{W}\|_{L^2(y^{\alpha},\C_{\Y})} + \| \partial_{yy} \mathcal{W}\|_{L^2(y^{\beta},\C_{\Y})} \lesssim \| w \|_{\mathbb{H}^{1+s}(\Omega)}.
\end{equation}

The following result shows that by considering \eqref{eq:truncgradflow} instead of \eqref{eq:gradflowextension} we only incur in an exponentially small error; compare with \cite[Lemma 3.3]{NOS}, \cite[Lemma 4.3]{NOS3} and \cite[Proposition 4.20]{NOS4}.

\begin{proposition}[exponential error estimate]
\label{prop:experror}
Let $\ve$ and $v$ be the solutions of \eqref{eq:gradflowextension} and \eqref{eq:truncgradflow}, respectively. Then, for $\Y\geq1$, we have
\[
  \| \tr (\ve - v) \|_{L^\infty(0,T;L^2(\Omega))}^2 + \| \GRAD( \ve - v ) \|_{L^2(0,T;L^2(y^\alpha,\C_\Y))}^2 \lesssim e^{-\Y/4}\frakD^2,
\]
where the hidden constant does not depend on neither $\ve$, $v$,  nor the problem data.
\end{proposition}
\begin{proof}
By a trivial zero extension we realize that the solution $v$ to problem \eqref{eq:truncgradflow} belongs to $\K(\C)$, then we can set $\phi = v$ in \eqref{eq:gradflowextension}. We would like to set $\phi=\ve$ in \eqref{eq:truncgradflow} but, although it satisfies the constraints, it is not an admissible test function, as it does not have a vanishing trace at $y = \Y$. For this reason, instead, we set $\phi = \rho \ve$ in \eqref{eq:truncgradflow}, where $\rho \in W^{1,\infty}(0,\infty)$ is the following smooth cutoff function:
\[
  \rho(y) = 1 \quad 0 \leq y \leq \frac\Y2, \qquad
  \rho(y) = \frac2\Y \left( \Y - y \right) \quad  \frac\Y2 < y < \Y, \qquad
  \rho(y) =0 \quad \Y \geq y.
\]
With these choices of test functions we add the ensuing inequalities to obtain
\[
 \frac12 \diff_t \| \tr (\ve - v) \|_{L^2(\Omega)}^2 + a(\ve, \ve - v ) \leq a_\Y( v, \rho \ve - v ).
\]
We now notice that
\[
  a(\ve, \ve- v) = a_\Y(\ve,\ve - v ) + \| \GRAD \ve \|_{L^2(y^{\alpha},\Omega\times(\Y,\infty))}^2 \geq a_\Y(\ve,\ve - v ),
\]
so that we obtain 
\begin{equation}
\label{eq:truncenergy}
  \frac12 \diff_t \| \tr (\ve - v) \|_{L^2(\Omega)}^2 + a_\Y(\ve - v, \ve - v ) \leq a_\Y(v, (\rho-1) \ve). 
\end{equation}
It remains then to bound the right hand side of \eqref{eq:truncenergy}. A straightforward computation reveals that if $y<\Y/2$ we have that $(\rho -1 )\ve \equiv 0$, otherwise
\[
  |\GRAD (\rho -1 )\ve|^2 \leq  2 \left( \frac{4}{\Y^2} \ve^2 + |\nabla \ve|^2\right),
\]
and thus
\[
  \| \GRAD (\rho-1)\ve \|_{L^2(y^\alpha,\C_\Y)}^2 
  \leq 2
  \left( \frac{4}{\Y^2}\int_{\Y/2}^\Y \int_{\Omega} y^{\alpha}|\ve|^2\diff x' \diff y +
  \int_{\Y/2}^\Y \int_{\Omega}y^{\alpha} |\nabla \ve|^2\diff x' \diff y\right). 
\]
Invoking a version of the Poincar\'e inequality \eqref{Poincare_ineq} based on the interval $[\Y/2,\Y]$, we conclude
$
  \| \GRAD (\rho-1)\ve \|_{L^2(y^\alpha,\C_\Y)}^2  \lesssim \| \GRAD \ve \|_{L^2(y^\alpha,\Omega\times(\Y/2,\Y))}^2.
$
We now use this estimate in \eqref{eq:truncenergy} and integrate in time to obtain
\begin{multline*}
  \| \tr(\ve -v)(t) \|_{L^2(\Omega)}^2 + \int_0^t \| \GRAD (\ve - v )(s) \|_{L^2(y^\alpha,\C_\Y)}^2 \diff s   \\ \leq
  \| \GRAD v \|_{L^2(0,T;L^2(y^\alpha,\C_\Y))} \| \GRAD \ve \|_{L^2(0,T;L^2(y^\alpha,\Omega\times(\Y/2,\infty)))}
  \lesssim e^{-\Y/4}\frakD^2,
\end{multline*}
where the last inequality follows from Lemma~\ref{lem:expdecay}, the fact that $\tr (\ve - v)_{|t=0} = 0$ and the stability estimate \eqref{eq:stab_ext_trun} of \eqref{eq:truncgradflow} in terms of $\frakD$. Since $t$ is arbitrary this implies the desired estimate.
\end{proof}

\section{Time discretization}
\label{sec:timediscr}

We now proceed with the time discretization of problem \eqref{eq:obstrong}. We could directly apply a suitable time discretization scheme to either \eqref{eq:gradflow}, \eqref{eq:varineq} or \eqref{eq:gradflowextension} and then argue,  for instance, by using the results of \cite{NSV}. However, with the exponential convergence result of Proposition~\ref{prop:experror} at hand we will consider the time discretization of the truncated problem \eqref{eq:truncgradflow} based on the implicit Euler method.

The discrete scheme computes the sequence $V^\dt \subset \K(\C_\Y)$, an approximation to the solution to problem \eqref{eq:truncgradflow} at each time step. We initialize the scheme by setting 
\begin{equation}
\label{eq:initial}
\tr V^0 = \ue_0, 
\end{equation}
and for $k = 0,\dots \mathcal{K}-1$, let $V^{k+1} \in \K(\C_\Y)$ be such that, for every $\phi \in \K(\C_\Y)$,
\begin{equation}
  \label{eq:semidicsflow}
  \begin{aligned}
  \left( \tr \frac{ \frakd V^{k+1} }\dt, \tr (V^{k+1} - \phi) \right)_{L^2(\Omega)} &+ a_\Y( V^{k+1}, V^{k+1} - \phi ) 
  \\ &\leq
  ( \fe^{k+1}, \tr (V^{k+1} - \phi) )_{\Ldeux} ,
  \end{aligned}
\end{equation}
where $\frakd$ is defined in Section~\ref{sec:notation} and $\fe^{k+1} = \dt^{-1} \int_{t_k}^{t_{k+1}} \fe \diff t \in L^2(\Omega)$. 

Existence and uniqueness of a solution to \eqref{eq:semidicsflow} follows directly from standard arguments on variational inequalities \cite{Brezis,KS}. The approximate solution to problem \eqref{eq:obstrong} is then defined by the sequence $U^\tau  \subset \Hs$ where
\begin{equation}
\label{discrete_Ufd}
U^\tau = \tr V^\tau.
\end{equation}

\begin{remark}[locality]\rm
The main advantage of scheme \eqref{eq:initial}--\eqref{eq:semidicsflow} is its local nature, which mimics that of problem \eqref{eq:gradflowextension}.
\end{remark}

Let us now show the stability of the scheme.

\begin{proposition}[stability]
Assume that $\ue_0 \in \K(\Omega)$ and $\fe \in L^2(0,T;L^2(\Omega))$, then $\hat{U}^\tau \in L^\infty(0,T;\Ldeux)$ and $\bar{V}^\tau \in L^2(0,T;\HL(y^\alpha,\C_\Y))$ uniformly in $\dt$.
\label{prop:stabsemidiscrete}
\end{proposition}
\begin{proof}
Set $\phi = \calH_\alpha \psi$ in \eqref{eq:semidicsflow}, where $\calH_\alpha$ is the $\alpha$-harmonic extension operator introduced in \eqref{eq:mathcalH}. Upon denoting $W^\tau = V^\tau - \calH_{\alpha} \psi$ we obtain
\begin{multline*}
  \left( \tr \frac{\frakd W^{k+1} }\tau, \tr W^{k+1} \right)_{\Ldeux} + a_\Y( W^{k+1}, W^{k+1} )
  \leq \\
  (f^{k+1}, \tr W^{k+1} )_\Ldeux + a_\Y( \calH_{\alpha} \psi, W^{k+1} ).
\end{multline*}
The Cauchy Schwartz inequality and summation over $k$ yields the result.
\end{proof}

The error analysis of \eqref{eq:initial}--\eqref{eq:semidicsflow} follows from the general theory presented in \cite{NSV}. To present it we introduce the error
\begin{equation}
\label{eq:E}
  E(v,V^{\tau}) =  \| \tr(v - \hat{V}^\dt ) \|_{L^\infty(0,T;\Ldeux)} + \| \GRAD( v - \bar{V}^\dt)\|_{L^2(0,T;L^2(y^\alpha,\C_\Y))},
\end{equation}
where $\bar{V}^\dt$ and $\hat{V}^\dt$ are defined in Section~\ref{sec:notation}. We also define
\begin{equation}
\label{eq:calE}
 \calE(\ue,U^{\tau}) =  \| \ue - \hat{U}^\dt  \|_{L^\infty(0,T;\Ldeux)} + \| \ue - \bar{U}^\dt \|_{L^2(0,T;\Hs)}.
\end{equation}

\begin{corollary}[error estimates in time I]
\label{col:semiimpl}
If $\ue_0 \in \K(\Omega)$ and $\fe \in L^2(0,T;\Ldeux)$, then the solutions $v$ of \eqref{eq:truncgradflow} and $V^{\tau}$ of \eqref{eq:initial}--\eqref{eq:semidicsflow} satisfy the uniform estimate
\[
  E(v,V^{\tau}) \lesssim \dt^{1/2}\left( \| \ue_0 \|_{\Hs} + \| \fe \|_{L^2(0,T;\Ldeux)} \right).
\]
If, on the other hand, we have that $\ue_0 \in \K(\Omega) \cap \polH^{2s}(\Omega)$, $\fe \in BV(0,T;\Ldeux)$ and $\fe^k = \fe_+(t_k)$, then
\[
  E(v,V^{\tau}) \lesssim \dt (\left\| \fe_{+}(0) - \Laps \ue_0 \right\|_{L^2(\Omega)} + \Var_{L^2(\Omega)} \fe  ).
\]
In these estimates the hidden constants do not depend on $v$, $V^{\tau}$ nor the problem data.
\end{corollary}
\begin{proof}
See Theorem~3.16 and Theorem~3.20 of \cite{NSV}.
\end{proof}

The following result combines Corollary~\ref{col:semiimpl}, the Caffarelli-Silvestre result \cite{CS:07}, the trace estimate \eqref{Trace_estimate} and the exponential error estimate of Proposition \ref{prop:experror}.

\begin{corollary}[error estimates in time II]
\label{col:semiimplII}
Assume that $\ue_0 \in \K(\Omega)$ and $\fe \in L^2(0,T;\Ldeux)$, then the solution $\ue$ of \eqref{eq:varineq} and the approximation $U^{\tau}$ defined by \eqref{discrete_Ufd} satisfy the uniform estimate
\[
  \calE(\ue,U^{\tau})\lesssim \dt^{1/2}\left( \| \ue_0 \|_{\Hs} + \| \fe \|_{L^2(0,T;\Ldeux)} \right) + e^{-\Y/8} \frakD.
\]
If we have that $\ue_0 \in \K(\Omega) \cap \polH^{2s}(\Omega)$, $\fe \in BV(0,T;\Ldeux)$ and $\fe^k = \fe_+(t_k)$, then
\[
  \calE(\ue,U^{\tau}) \lesssim \dt (\left\| \fe_{+}(0) - \Laps \ue_0 \right\|_{L^2(\Omega)} + \Var_{L^2(\Omega)} \fe  ) + e^{-\Y/8} \frakD,
\]
where in both estimates the hidden constants depend solely on the problem data.
\end{corollary}
\begin{proof}
The definition of $\calE(\ue,U^\tau)$, given in \eqref{eq:calE}, the Caffarelli-Silvestre extension result $\tr \ve = \ue$ \cite{CS:07,NOS} and estimate \eqref{Trace_estimate} yield $\calE(\ue,U^\tau) \lesssim E(\ve,V^\tau)$, where $E$ is defined in \eqref{eq:E}. Notice now that $E$ is sublinear in its first argument, so that
\[
  \calE(\ue,U^\tau) \lesssim E(v,V^\tau) + \| \tr(\ve - v)  \|_{L^\infty(0,T;\Ldeux)} + \| \ve - v \|_{L^2(0,T;L^2(y^{\alpha},\C_\Y))},
\]
The result now follows from combining Corollary~\ref{col:semiimpl} and Proposition~\ref{prop:experror}.
\end{proof}

\section{Space discretization}
\label{sec:space}

The results of previous sections are important in two aspects. First, we were able to replace the original parabolic fractional obstacle problem \eqref{eq:obstrong} (or any of its variants discussed in \S\ref{sub:fracobstacle}) by an equivalent one that involves a local operator \eqref{eq:gradflowextension} and is posed on the semi-infinite cylinder $\C$. Then, we considered a truncated version \eqref{eq:truncgradflow} of our problem, that is posed on the \emph{bounded domain} $\C_{\Y}$, while just incurring in an exponentially small error in the process. This is important because we shall discretize in space using first-degree tensor product finite elements. Section~\ref{sec:timediscr} presents a first order discretization in time and applies the general theory of discretizations of nonlinear evolution equations \cite{NSV} to provide an error analysis.

It remains then to discretize in space and to study its effect. We will follow \cite{NOS4,NOS3,NOS}, where it is shown that $\calU$, solution of \eqref{eq:extension}, possesses a singularity as $y\downarrow0$, so that the use of anisotropic meshes in the extended direction $y$ is imperative if one wishes to obtain a quasi-optimal approximation error. The latter combines asymptotic properties of Bessel functions with polynomial interpolation theory on weighted Sobolev spaces \cite{NOS2}, which is valid for tensor product elements that exhibit a large aspect ratio in $y$. These references also show how to exploit the tensor product structure of $\C_\Y$ to design such a mesh. For convenience we recall this construction.

Let $\T_\Omega = \{K\}$ be a conforming and shape regular triangulation of $\Omega$ into cells $K$ that are isoparametrically equivalent to either a simplex or a cube \cite{CiarletBook,Guermond-Ern}. We denote by $\sigma_{\Omega}$ the shape regularity constant of $\T_\Omega$. Let $\calI_\Y = \{I\}$ be a partition of $[0,\Y]$ with mesh points
\begin{equation}
\label{eq:gradedmesh}
  y_j = \left( \frac{j}M \right)^\gamma \Y, \quad j = 0,\ldots,M, \quad \gamma > \frac3{1-\alpha}=\frac3{2s}>1.
\end{equation}
We then construct a mesh of the cylinder $\C_\Y$ by $\T_\Y = \T_\Omega \otimes \calI_\Y$, \ie each cell $T \in \T_\Y$ is of the form $T = K \times I$ where $K \in \T_\Omega$ and $I \in \calI_\Y$. Notice that, by construction, $\# \T_\Y = M \#\T_\Omega$. When $\T_\Omega$ is quasiuniform with $\# \T_\Omega \approx M^d$ we have $\# \T_\Y \approx M^{d+1}$ and, if $h_{\T_\Omega} = \max \{ \diam(K) : K \in \T_\Omega \}$, then $M \approx h_{\T_\Omega}^{-1}$.

Having constructed the mesh $\T_\Y$ we define the finite element space
\begin{equation*}
\label{eq:defofFE}
  \polV(\T_\Y) := \left\{ W \in C^0(\bar\C_\Y): W_{|T} \in \calP(K) \otimes \polP_1(I) \ \forall T \in \T_\Y, \ W_{|\Gamma_D} =0 \right\},
\end{equation*}
where, if $K$ is isoparametrically equivalent to a simplex, $\calP(K)=\polP_1(K)$ \ie the set of polynomials of degree at most one.
If $K$ is a cube $\calP(K) = \polQ_1(K)$, that is, the set of polynomials of degree at most one in each variable.

We remark that, owing to \eqref{eq:gradedmesh}, the meshes $\T_\Y$ are not shape regular but satisfy: if $T_1 = K_1 \times I_1$ and $T_2 = K_2 \times I_2$ are neighbors, then there is $\sigma>0$ such that
\[
 h_{I_1} \leq \sigma h_{I_2}, \qquad h_I = |I|.
\]
While this is crucial to capture the singularities present in the solution, it also requires the development of anisotropic error estimates on Muckenhoupt weighted Sobolev spaces as detailed in \cite{NOS,NOS2}.

\subsection{Fully discrete scheme}
\label{sub:descscheme}
To describe and analyze the fully discrete scheme we must introduce an interpolation operator that preserves positivity of traces. In what follows we assume that there is an operator $\Pi_{\T_\Y}: L^1(\C_\Y) \to \polV(\T_\Y)$ that verifies:
\begin{enumerate}[$\bullet$]
  \item \emph{Locality}. If $w$ is such that $\tr w$ makes sense and $\vero'$ is a vertex in $\T_\Omega$, then $\Pi_{\T_\Y} w(\vero',0)$ depends only on the values of $w$ in an $\Omega$-neighborhood of $\vero'$.
  
  \item \emph{Positivity preserving}. If $w$ is such that $\tr w$ makes sense then
  \begin{equation}
  \label{eq:Pipositive}
    \tr w \geq 0 \implies \Pi_{\T_\Y} w \geq 0.
  \end{equation}
  
  \item \emph{Stability}. If $w$ is $\alpha$-harmonic and $w \in H^2(y^\beta,\C_\Y)\cap \HL(y^\alpha,\C_\Y)$ then 
  \begin{equation}
  \label{eq:Pistable}
    \| \Pi_{\T_\Y} w \|_{\HLn(y^\alpha,\C_\Y)} \lesssim \| w \|_{H^2(y^\beta, \C_\Y)} + \| w \|_{\HLn(y^\alpha,\C_\Y)}
  \end{equation}
  
  \item \emph{Approximation}. If $w$ is $\alpha$-harmonic and $w \in H^2(y^\beta,\C_\Y)\cap \HL(y^\alpha,\C_\Y)$ then 
  \begin{equation}
  \label{eq:Piapprox}
    \lim_{\# \T_\Y \to \infty} \| w - \Pi_{\T_\Y} w \|_{\HLn(y^\alpha,\C_\Y)} = 0.
  \end{equation}
  
  \item \emph{Superapproximation}. For $\delta \in [0,2]$ we have that if $\tr w \in \polH^\delta(\Omega)$, then
  \begin{equation}
  \label{eq:Pisuperapprox}
    \| \tr( w - \Pi_{\T_\Y} w ) \|_{\Ldeux} \lesssim h_{\T_\Omega}^\delta \| w \|_{\polH^\delta(\Omega)},
  \end{equation}
  where the hidden constant is independent of $\delta$, $\T_\Omega$ and $w$.
  
\end{enumerate}
An example of such a construction is presented in Section~\ref{sub:interpolant}.

Let us now describe the scheme. We define
\begin{equation*}
\label{eq:defofKh}
\K(\T_\Y) = \left\{ W \in \polV(\T_\Y) : \tr W \geq \tr \Pi_{\T_\Y} \calH_\alpha \psi \ \mae \Omega \right\},
\end{equation*}
where $\calH_\alpha$ is the $\alpha$-harmonic extension operator introduced in \eqref{eq:mathcalH}.
The fully discrete scheme computes the sequence $V^{\tau}_{\T_{\Y}} \subset \K(\T_\Y)$, an approximation of the solution to \eqref{eq:truncgradflow} at each time step. We initialize the scheme by setting 
\begin{equation}
\label{eq:fdinit}
V_{\T_\Y}^0 = \Pi_{\T_\Y} \calH_\alpha \ue_0.
\end{equation}
For $k=0,\cdots,\calK-1$, $V_{\T_\Y}^{k+1} \in \K(\T_\Y)$ solves
\begin{multline}
\label{eq:fdscheme}
  \left( \tr \frac{\frakd V_{\T_\Y}^{k+1}}\dt, \tr(V_{\T_\Y}^{k+1} - W) \right)_{L^2(\Omega)}  + a_\Y(V_{\T_\Y}^{k+1}, V_{\T_\Y}^{k+1} - W)
  \\
  \leq \left( \fe^{k+1}, \tr(V_{\T_\Y}^{k+1}- W) \right)_{L^2(\Omega)} \qquad \forall W \in \K(\T_\Y) .
\end{multline}
Standard results on variational inequalities yield existence and uniqueness of $V_{\T_\Y}^{k+1}$ for $k=1,\dots,\calK-1$. To obtain an approximate solution to the parabolic fractional obstacle problem \eqref{eq:obstrong}, we define the sequence $U^{\tau}_{\T_{\Omega}} \subset \Hs $ by
$U^{\tau}_{\T_{\Omega}} = \tr V^{\tau}_{\T_{\Y}}
$.

\begin{remark}[properties of the scheme]\rm
The main advantage of \eqref{eq:fdscheme}--\eqref{eq:fdinit} is that provides an approximate solution to the fractional obstacle problem \eqref{eq:obstrong} based on solving the local evolution variational inequality \eqref{eq:fdscheme}. Its implementation is simple and requires standard components of a finite element algorithm.
\end{remark}

We note that, if $\ue_0 \in \polH^{1+s}(\Omega)$, the continuity \eqref{eq:Pistable} of the operator $\Pi_{\T_{\Y}}$ implies
\begin{equation}
\label{eq:IDatabdd}
  \| \GRAD V_{\T_{\Y}}^0 \|_{L^2(y^\alpha,\C_\Y)} \lesssim \| \ue_0 \|_{\polH^{1+s}(\Omega)}.
\end{equation}
Indeed, the regularity results of \cite[Theorem 2.7]{NOS} show that, if $w \in \polH^{1+s}(\Omega)$, then $\partial_{yy} \calH_\alpha w \in L^2(y^\beta,\C_\Y)$ and $\nabla \nabla_{x'} \calH_\alpha w \in L^2(y^\alpha,\C_\Y)$; see also \eqref{eq:reg_estimates}.

Let us now present an error analysis for \eqref{eq:fdscheme}. We will do so under different assumptions on the problem data, thus obtaining different rates according to the smoothness properties of the solution. We introduce the errors
\begin{equation}
\label{e}
e^\dt = e^\dt(V^{\dt},V_{\T_{\Y}}^{\tau}) = V^\dt - V_{\T_{\Y}}^{\tau},
\end{equation}
where $V^{\dt}$ and $V_{\T_{\Y}}^{\tau}$ solve the problems \eqref{eq:initial}--\eqref{eq:semidicsflow} and \eqref{eq:fdinit}--\eqref{eq:fdscheme} respectively, and 
\begin{equation}
\label{vare}
\vare^\dt = \vare^\dt(V^{\tau}) = V^\dt - \Pi_{\T_\Y} V^\dt.
\end{equation}

\subsection{Analysis with minimal regularity}
\label{sub:minreg}
Here we only assume that the right hand side satisfies $\fe \in L^2(0,T;\Ldeux)$, $\psi \in \Hs$ and $\psi \leq \ue_0 \in \polH^{1+s}(\Omega)$. As a first result we obtain an a priori estimate for $V^\dt_{\T_\Y}$.

\begin{lemma}[a priori estimates on $V^\dt_{\T_\Y}$]
\label{lem:aprioriV}
If $\fe \in L^2(0,T;\Ldeux)$, $\psi \in \Hs$ and $\psi \leq \ue_0 \in \polH^{1+s}(\Omega)$, then the sequence $V^\dt_{\T_\Y}$, solution to \eqref{eq:fdinit}--\eqref{eq:fdscheme}, satisfies
\begin{multline*}
  \sum_{k=1}^{\calK} \| \tr \frakd V_{\T_{\Y}}^k \|_{\Ldeux}^2 + \dt \| \GRAD V_{\T_{\Y}}^{\calK} \|_{L^2(y^\alpha,\C_\Y)}^2 
  + \dt \sum_{k=1}^{\calK} \|\GRAD \frakd V_{\T_{\Y}}^{k} \|_{L^2(y^\alpha,\C_\Y)}^2 \\
  \lesssim \dt \left[ \| \fe \|_{L^2(0,T;\Ldeux)}^2 + \| \ue_0\|_{\polH^{1+s}(\Omega)}^2 \right],
\end{multline*}
where the hidden constant does not depend on neither $v$, $V_{\T_{\Y}}^\dt$, nor the problem data.
\end{lemma}
\begin{proof}
Set $W = V_{\T_{\Y}}^k$ in \eqref{eq:fdscheme} and multiply the obtained result by $\dt$. Using the identity $2a(a-b) = a^2 - b^2 +(a-b)^2$, with $a,b \in \Real$, we derive
\begin{multline*}
  \| \tr \frakd V_{\T_{\Y}}^{k+1} \|_{\Ldeux}^2 + 
  \frac{\dt}{2\ds} \left[ \frakd\|\GRAD V_{\T_{\Y}}^{k+1} \|_{L^2(y^\alpha,\C_\Y)}^2 + \|\GRAD \frakd V_{\T_{\Y}}^{k+1} \|_{L^2(y^\alpha,\C_\Y)}^2 \right] \leq \\
  \dt \| \fe^{k+1} \|_{\Ldeux} \| \tr \frakd V_{\T_{\Y}}^{k+1} \|_{\Ldeux} \leq
  \frac12 \dt^2 \| \fe^{k+1} \|_{\Ldeux}^2 + \frac12 \| \tr \frakd V_{\T_{\Y}}^{k+1} \|_{\Ldeux}^2.
\end{multline*}
Adding this inequality over $k=0,\ldots,\calK-1$ yields
\begin{multline*}
\frac12\sum_{k=1}^{\calK} \| \tr \frakd V_{\T_{\Y}}^k \|_{\Ldeux}^2 + \frac\dt2 \|\GRAD V_{\T_{\Y}}^{\calK} \|_{L^2(y^\alpha,\C_\Y)}^2 + \frac\dt2 \sum_{k=1}^\calK \|\GRAD \frakd V_{\T_{\Y}}^k \|_{L^2(y^\alpha,\C_\Y)}^2 \lesssim \\ \frac\dt2 \left[ \| \fe \|_{L^2(0,T;\Ldeux)}^2 + \|\GRAD V_{\T_{\Y}}^0\|_{L^2(y^\alpha,\C_\Y)}^2 \right].
\end{multline*}
The assumptions on $\fe$ and $\ue_0$ imply, in light of \eqref{eq:IDatabdd}, the asserted estimate.
\end{proof}

With these a priori estimates we can provide a first error analysis. Notice that by means of the change of variable $\ue \leftarrow \ue - \psi$ one can assume that the obstacle $\psi \equiv 0$. In this case then we have that $\K(\T_\Y) \subset \K(\C_\Y)$.

\begin{theorem}[error analysis with minimal regularity]
\label{thm:firstapriori}
If $\fe \in L^2(0,T;\Ldeux)$, $\psi \equiv 0$ and $0 \leq \ue_0 \in \polH^{1+s}(\Omega)$, then
\begin{multline*}
  \| \tr \hat{e}^\dt \|_{L^\infty(0,T;\Ldeux)}^2 + \| \GRAD \bar{e}^\dt \|_{L^2(0,T;L^2(y^\alpha,\C_\Y))}^2 \lesssim \\
  \| \tr e^0 \|_{\Ldeux}^2 + 
  \left\| \GRAD \bar{\vare}^\dt \right\|_{L^2(0,T;L^2(y^\alpha,\C_\Y))},
\end{multline*}
where the hidden constant depends only on $\|\fe\|_{L^2(0,T;\Ldeux)}$ and $\| \ue_0 \|_{\polH^{1+s}(\Omega)}$.
\end{theorem}
\begin{proof}
Set $\phi = V_{\T_{\Y}}^{k+1}$ in \eqref{eq:semidicsflow} and $W = \Pi_{\T_\Y} V^{k+1}$ in \eqref{eq:fdscheme} and add the resulting inequalities to arrive at
\begin{multline*}
  \left( \tr \frakd e^{k+1}, \tr e^{k+1} \right)_{L^2(\Omega)} + \dt \| \GRAD e^{k+1} \|_{L^2(y^\alpha,\C_\Y)}^2 \leq \\
  - \left( \tr \frakd V_{\T_{\Y}}^{k+1}, \tr \vare^{k+1} \right)_{L^2(\Omega)} - \dt a_\Y(V_{\T_{\Y}}^{k+1},\vare^{k+1}) 
   + \dt \left( \fe^{k+1}, \tr \vare^{k+1} \right)_{L^2(\Omega)},
\end{multline*}
where $e^{\tau}$ and $\vare^{\tau}$ are defined by \eqref{e} and \eqref{vare}, respectively. Added over $k=0,\ldots,\ell-1$, this inequality yields
\begin{multline*}
  \| \tr e^\ell \|_{\Ldeux}^2 + \dt \sum_{k=1}^\ell \| \GRAD e^k \|_{L^2(y^\alpha,\C_\Y)}^2
  \lesssim
  \| \tr e^0 \|_{\Ldeux}^2  +\\
  \left[
    \dt^{-\tfrac{1}{2}} \left( \sum_{k=0}^{\calK-1} \| \tr \frakd V^{k+1}_{\T_\Y} \|_{\Ldeux}^2 \right)^{\tfrac{1}{2}} + \| \GRAD \bar{V}^\dt_{\T_\Y} \|_{L^2(0,T;L^2(y^\alpha,\C_\Y))}
    + \| \fe \|_{L^2(0,T;\Ldeux)}
  \right] \times \\
  \| \GRAD \bar{\vare}^\dt \|_{L^2(0,T;L^2(y^\alpha,\C_\Y))}.
\end{multline*}

Notice now that Lemma~\ref{lem:aprioriV} implies that
$
  \dt^{-1} \sum_{k=0}^{\calK-1} \| \tr \frakd V^{k+1}_{\T_\Y} \|_{\Ldeux}^2 \lesssim 1
$
and
$
  \| \GRAD \bar{V}^\dt_{\T_\Y} \|_{L^\infty(0,T;L^2(y^\alpha,\C_\Y))} \lesssim 1
$.
These estimates allow us to conclude.
\end{proof}

\begin{remark}[suboptimal estimate]\rm
Notice that, in the conclusion of Theorem~\ref{thm:firstapriori}, while the terms on the left hand side are squared, the interpolation error on the right is not. Therefore, even if the operator $\Pi_{\T_\Y}$ exhibited optimal approximation properties, this estimate is suboptimal in space. Nevertheless, from this result one can conclude convergence for rather general initial data $\ue_0$ and forcing term $\fe$	.
\end{remark}

\subsection{Analysis with regularity}
\label{sub:maxreg}

The results of \cite{MR3100955} show that, if $\ue_0 = \psi \in C^2(\bar\Omega)$ with $\Laps \psi \in C^{0,1-s}(\Omega)$ and $0 \leq \fe \in C^1((0,T],C^{0,1-s}(\Omega))$, then $\ue$ satisfies:
\begin{equation}
\label{eq:regu1}
  \diff_t \ue, \Laps \ue \in \logLip((0,T],C^{0,1-s}(\bar\Omega)) \quad s \leq \frac13
\end{equation}
and
\begin{equation}
\label{eq:regu2}
  \diff_t\ue, \Laps \ue \in C^{0,\frac{1-s}{2s}}((0,T],C^{0,1-s}(\bar\Omega)) \quad s > \frac13.
\end{equation}
These results, however, account only for the regularity of $\ue$, \ie the regularity of $\tr \ve$. For the elliptic obstacle problem \cite[Theorem 6.4]{AllenPetrosyan} studies the regularity of the solution over the cylinder $\calC_\Y$ and, on the basis of their findings, we shall assume that
\begin{equation}
\label{eq:regext}
  s \leq \frac12 \Rightarrow \ve \in C^{0,2s}(\calC_\Y); \qquad
  s > \frac12 \Rightarrow \ve \in C^{1,2s-1}(\calC_\Y).
\end{equation}
Let us now present, under these improved regularity conditions, an error analysis.

\begin{theorem}[analysis with regularity]
\label{thm:optapriori}
Assume that $\ue_0$, $\psi$ and $\fe$ are such that \eqref{eq:regu1}--\eqref{eq:regext} hold for $\ve$ and $\hat{V}^\dt$ uniformly in $\dt$.
If $\#\T_\Y \approx M^{d+1}$, then we have
\[
E(\ve,V^\dt_{\T_\Y}) \lesssim
  \dt + |\log M|^s \left[ M^{-1} + \| \GRAD \bar{\vare} \|_{L^\infty(0,T;L^2(y^\alpha,\C_\Y))} + \frac{ M^{-(1+s)} }{\tau^{1/2}}  \right],
\]
where the hidden constant depends only on the problem data.
\end{theorem}
\begin{proof}
The results of Proposition~\ref{prop:experror} and Corollary~\ref{col:semiimpl} reduce the analysis to estimating the difference $e^\dt = V^\dt - V^\dt_{\T_\Y}$.
Using the well-known indentity $2a(a-b) = a^2 - b^2 +(a-b)^2$ we derive
\begin{multline*}
  \frac12 \left( \frakd \| \tr e^{k+1}\|_{\Ldeux}^2 + \| \tr \frakd e^{k+1} \|_{\Ldeux}^2 \right) + \frac\dt\ds \| \GRAD e^{k+1} \|_{L^2(y^\alpha,\C_\Y)}^2  \\
  = \left( \tr \frakd e^{k+1}, \tr e^{k+1} \right)_{\Ldeux} + \frac\dt\ds \| \GRAD e^{k+1} \|_{L^2(y^\alpha,\C_\Y)}^2.
\end{multline*}
Therefore, invoking \eqref{e} and \eqref{vare}, we arrive at
\begin{multline*}
   \left( \tr \frakd e^{k+1}, \tr e^{k+1} \right)_{\Ldeux} = \left( \tr \frakd e^{k+1}, \tr \vare^{k+1} \right)_{\Ldeux} \\
   + \left( \tr \frakd e^{k+1}, \tr (\Pi_{\T_\Y} V^{k+1} - V^{k+1}_{\T_\Y}) \right)_{\Ldeux} 
   \leq \frac12 \| \tr \frakd e^{k+1} \|_{\Ldeux}^2 \\
   + \frac12 \|\tr \vare^{k+1} \|_{\Ldeux}^2 
  + \left( \tr \frakd e^{k+1}, \tr ( \Pi_{\T_\Y} V^{k+1} - V^{k+1}_{\T_\Y}) \right)_{\Ldeux}
\end{multline*}
and
\begin{multline*}
  \frac1\ds \| \GRAD e^{k+1} \|_{L^2(y^\alpha,\C_\Y)}^2 = a_\Y(e^{k+1}, \vare^{k+1} ) + a_\Y( e^{k+1}, \Pi_{\T_\Y} V^{k+1} - V^{k+1}_{\T_\Y} ) \\
  \leq \frac1{2\ds} \| \GRAD e^{k+1} \|_{L^2(y^\alpha,\C_\Y)}^2 + \frac1{2\ds} \| \GRAD \vare^{k+1} \|_{L^2(y^\alpha,\C_\Y)}^2
  + a_\Y( e^{k+1}, \Pi_{\T_\Y} V^{k+1} - V^{k+1}_{\T_\Y} ).
\end{multline*}
Consequently,
\begin{multline*}
  \frakE^{k+1}:=\frac12 \frakd \| \tr e^{k+1} \|_{\Ldeux}^2 + \frac\dt{2\ds} \| \GRAD e^{k+1} \|_{L^2(y^\alpha, \C_\Y )}^2 \\
  \lesssim \frac12 \| \tr \vare^{k+1} \|_{\Ldeux}^2 +\frac\dt{2\ds} \| \GRAD \vare^{k+1} \|_{L^2(y^\alpha, \C_\Y )}^2
  \\
  + \left( \tr \frakd e^{k+1}, \tr ( \Pi_{\T_\Y} V^{k+1} -  V^{k+1}_{\T_\Y} ) \right)_{\Ldeux}  +
  \dt a_\Y(e^{k+1}, \Pi_{\T_\Y} V^{k+1} - V^{k+1}_{\T_\Y} )
\end{multline*}
Since $\Pi_{\T_\Y} V^{k+1} \in \K(\T_\Y)$, we use the scheme \eqref{eq:fdscheme} with $W = \Pi_{\T_\Y} V^{k+1}$ to derive
\begin{multline*}
  \frakE^{k+1} \lesssim 
  \| \tr \vare^{k+1} \|_{\Ldeux}^2 + \dt \| \GRAD \vare^{k+1} \|_{L^2(y^\alpha, \C_\Y )}^2 + \\
  \left( \tr \frakd V^{k+1}, \tr (\Pi_{\T_\Y} V^{k+1} - V^{k+1}_{\T_\Y}) \right)_{\Ldeux} + \dt a_\Y( V^{k+1}, \Pi_{\T_\Y} V^{k+1} - V^{k+1}_{\T_\Y} )\\
  - \dt \left( \fe^{k+1}, \tr (\Pi_{\T_\Y} V^{k+1} - V^{k+1}_{\T_\Y}) \right)_{\Ldeux}.
\end{multline*}

The smoothness assumptions on $V^\dt$ allow us to integrate by parts to obtain
\begin{multline*}
  a_\Y(V^{k+1}, \Pi_{\T_\Y} V^{k+1} - V^{k+1}_{\T_\Y} ) = 
  - \frac1\ds \int_{\C_\Y} \DIV(y^\alpha \GRAD V^{k+1})(\Pi_{\T_\Y} V^{k+1} - V^{k+1}_{\T_\Y} ) \diff x' \diff y \\
  + \frac1\ds \left( \partial_{\nu^\alpha} V^{k+1}, \tr \Pi_{\T_\Y} V^{k+1} - \tr V^{k+1}_{\T_\Y} \right)_{\Ldeux}.
\end{multline*}
Since $\calZ^{k+1} = \dt^{-1} \tr \frakd V^{k+1} - \fe^{k+1} + \ds^{-1} \partial_{\nu^\alpha} V^{k+1} \geq 0$, the regularity assumptions imply, in particular, that $\hat{\calZ}^\dt \in C([0,T];C^{0,1-s}(\Omega))$. Therefore
\begin{align*}
  \frakE^{k+1} &\lesssim \| \tr \vare^{k+1} \|_{\Ldeux}^2 + \dt \| \GRAD \vare^{k+1} \|_{L^2(y^\alpha, \C_\Y )}^2 \\ &+ 
  \dt \int_{\Omega \times \{0\}} \calZ^{k+1} \tr(\Pi_{\T_\Y} V^{k+1} - V^{k+1}_{\T_\Y} ) \diff x'.
\end{align*}

We proceed now as in \cite[Theorem 4.24]{NOS4} and realize that it suffices to consider
\[
 \sum_{K \in \T_{\Omega}} \int_{K \times \{0\}} \calZ^{k+1} ( \tr( \Pi_{\T_\Y}V^{k+1} - \Pi_{\T_\Y}\calH_\alpha \psi) - \tr( V^{k+1} - \calH_\alpha \psi) ) \diff x' 
 =  \sum_{K \in \T_{\Omega}} \mathcal{I}(K).
\]
We analyze separately the cells $K \in \T_{\Omega}$ according to the value of $\tr( V^{k+1} - \calH_\alpha \psi) $.
\begin{enumerate}[$\bullet$]
 \item \emph{$\tr( V^{k+1} - \calH_\alpha \psi) > 0 $ in a neighborhood of $K$.} In this situation $\calZ^{k+1} = 0$, and thus $\mathcal{I}(K)$ vanishes.
 \item \emph{$\tr( V^{k+1} - \calH_\alpha \psi) = 0 $ in a neighborhood of $K$.} The linearity of $\Pi_{\T_{\Y}}$ yields that $\mathcal{I}(K)$ vanishes.
 \item \emph{$\tr( V^{k+1} - \calH_\alpha \psi) $ is not identically zero nor strictly positive in a neighborhood of $K$.} In this case, either $\tr( \Pi_{\T_\Y}V^{k+1} - \Pi_{\T_\Y}\calH_\alpha\psi) = 0$ or $\calZ^{k+1}=0$. If $K$ is such a cell, then there is $x_0' \in K$ where $\tr V^{k+1}(x_0') = \psi(x_0')$ so that the smoothness assumptions on $\psi$ and the regularity results of \cite{MR3100955} allow us to conclude the growth estimate
$
  0 \leq \tr V^{k+1}(x') - \psi(x') \lesssim h_{\T_\Omega}^{1+s} \quad \forall x' \in K.
$
By the same reasoning
$
  0 \leq \calZ^{k+1}(x') \lesssim h_{\T_\Omega}^{1-s}.
$
Then, $\mathcal{I}(K) \lesssim h_{\T_\Omega}^2$.
\end{enumerate}
Collecting the derived estimates we obtain
\begin{equation*}
\label{frakE_estimate}
  \frakE^{k+1} \lesssim \| \tr \vare^{k+1} \|_{\Ldeux}^2 + \dt \| \GRAD \vare^{k+1} \|_{L^2(y^\alpha, \C_\Y )}^2 + \dt h_{\T_\Omega}^2 .
\end{equation*}
Add this expression over $k$. Using Proposition~\ref{prop:stabsemidiscrete} and the regularity results \eqref{eq:regu1} and \eqref{eq:regu2} we have that $\tr \bar{V}^\tau \in L^\infty(0,T;\polH^{1+s}(\Omega))$. Therefore, using the superapproximation of traces of the operator $\Pi_{\T_\Y}$ \eqref{eq:Pisuperapprox} we obtain
\[
  \sum_{k=1}^\calK \| \tr \vare^k \|_{\Ldeux}^2 \lesssim \frac1\tau  \| \tr\vare^\tau \|_{L^\infty(0,T;\Ldeux)}^2
  \lesssim \frac{ h_{\T_\Omega}^{2(1+s)} }\tau.
\]
To conclude, we recall that $h_{\T_\Omega} \approx(\# \T_\Y)^{-1/(d+1)}$.
\end{proof}

\section{Positivity preserving interpolation over anisotropic meshes}
\label{sub:interpolant}

The error analysis that was presented in previous sections relied on energy arguments, and for that we needed to choose suitable test functions in the semidiscrete and discrete schemes. This brings forth the need for a positivity preserving interpolant.

The construction of positivity preserving approximation operators has a rich history in approximation theory. The classical convergence \cite{MR0057472} and impossibility \cite{MR0089939} results of P.P.~Korovkin come immediately to mind in this respect. In the finite element literature, a positivity preserving interpolant was constructed in \cite{MR1742264} and it was later showed in \cite{MR1933037} that it cannot be of order higher than one.

The operator of \cite{MR1742264} is analyzed under the assumption that the mesh is shape regular. In our setting we need an interpolant that preserves constraints of traces and that exhibits suitable approximation properties in weighted spaces and over anisotropic meshes. This makes the extension of the ideas of \cite{MR1742264} not straightforward, if at all possible. For this reason, we will restrict our attention to the case $s>3/8$ and, by combining the ideas developed in the construction of the positivity preserving interpolation operator of \cite{MR1742264} with the quasi-interpolation operator analyzed in \cite{NOS,NOS2}, construct a positivity preserving operator on anisotropic meshes that possesses suitable approximation properties on $\alpha$-harmonic functions.

Let us present a slight modification of the quasi-interpolation operator $\calL_{\T_\Y}$ of \cite{NOS,NOS2}. To do so, we introduce some notation and terminology. Given $\T_\Y$, we denote by $\N$ the set of its nodes and by $\N_{\textrm{in}}$ the set of its interior and Neumann nodes. For each vertex $\vero \in \N$, we write $\vero= (\vero',\vero'')$, where $\vero'$ corresponds to a node of $\T_\Omega$, and $\vero''$ corresponds to a node of the discretization in the extended dimension. We define $h_{\vero'}= \min\{h_K: K \in \T_\Omega, \ \vero' \ni K\}$, and $h_{\vero''} = \min\{h_I: I \in \calI_\Y, \ \vero'' \ni I\}$. Given $\vero \in \N$, the \emph{star} or patch around $\vero$ is defined as
$
  S_{\vero} = \cup_{T \ni \vero  } T,
$
and, for $T \in \T_{\Y}$, we define its \emph{patch} as
$
  S_T = \cup_{\vero \in T} S_\vero.
$
We set $\N_\Omega = \{ \vero': (\vero',\vero'') \in \N_{\textrm{in}} \}$.

Let $\mu_1 \in C^{\infty}_0(\mathbb{R}^{n})$ be such that $\supp \mu_1 \subset B_{r}$, where $B_r$ denotes the ball in $\R^n$ centered at zero and with radius $r \leq 1/\sigma_{\Omega}$; moreover, we require that $\mu_1 \geq 0$, $\int \mu_1(x') \diff x' = 1$ and that $\mu_1$ has vanishing first order moments, \ie $\int \mu_1(x') x^i \diff x' = 0$ for all $i=1,\ldots, d$. Let $\mu_2 \in C_0^{\infty}(\mathbb{R})$ be such that $\supp \mu_2 \subset (0,r_{\Y})$, where $r_{\Y} \leq 1/\sigma$ and $\int \mu_2(y) \diff y = 1$. We then define $\mu(x',y):= \mu_1(x') \mu_2(y)$, which satisfies $\int \mu \diff x' \diff y = 1$ and $\supp \mu \subset B_r \times (0,r_{\Y})$. 
For $\vero \in \N_{\textrm{in}}$ we define
\[
  \mu_{1,\vero'}(x') = \frac1{h_{\vero'}^d} \mu_1\left(\frac{x'-\vero'}{h_{\vero'}} \right),
  \quad
  \mu_{2,\vero''}(y) = \frac1{  h_{\vero''} }  \mu_2\left( \frac{y-\vero''}{h_{\vero''}} \right),
\]
and
$
  \mu_{\vero}(x',y) = \mu_{1,\vero'}(x') \mu_{2,\vero''}(y)
$.
We note that $\textrm{supp}~\mu_{\vero} \subset S_{\vero}$ and $\int_{S_{\vero}} \mu_{\vero} \diff x' \diff y = 1$ for any node $\vero \in \N_{\textrm{in}}$.

Given a function $w \in L^1(\C_\Y)$ and a node $\vero$ in $\N_{\textrm{in}}$, the regularized Taylor polynomial of first degree of $w$ about $\vero$ is defined as follows: 
\begin{equation*}
\label{p1average}
  w_{\vero}(z) = \int_{S_{\vero}} P(x,z) \mu_{\vero}(x) \diff x,
\end{equation*}
where $P$ denotes the Taylor polynomial of degree one in the variable $z$ of the function $w$ about the point $x$, \ie
$ P(x,z) = w(x) + \nabla w(x) \cdot (z - x)$.

If $\Lambda_{\vero}$ denotes the Lagrange basis function associated with the node $\vero$ in the discretization $\T_{\Y}$, we then define the averaged interpolant $\mathcal{L}_{\T_{\Y}} w$ as follows:
\begin{equation}
\label{eq:L}
  \mathcal{L}_{\T_{\Y}} w = \sum_{\vero \in \N_{\textrm{in}}} w_{\vero}(\vero) \Lambda_{\vero}.
\end{equation}
$\mathcal{L}_{\T_{\Y}}$ is linear, stable and possesses optimal approximation properties in $\HL(y^\alpha,\calC_\Y)$. It is well suited for anisotropic meshes. The only difference between this construction and that of \cite{NOS,NOS2} is in the particular choice of the weighting function $\mu$.

We now turn to the construction of \cite{MR1742264} and slightly modify it to suit our purposes. Given $\phi \in L^1(\Omega)$ we define the interpolation operator $\mathcal{R}_{\T_{\Omega}}: L^1(\Omega) \rightarrow \polV(\T_{\Omega})$ by
\begin{equation}
\label{eq:CH_operator}
\mathcal{R}_{\T_{\Omega}}\phi = \sum_{\vero' \in \N_{\Omega}} \left( \int_{S_{\vero'}'} \mu_{1,\vero'}(z') \phi(z') \diff z' \right) \Lambda_{\vero'},
\end{equation}
where $S_{\vero'}'$ denotes the star (in $\Omega$) around $\vero'$ and $\Lambda_{\vero'}$ denotes the Lagrange basis function associated with node $\vero'$ in $\T_{\Omega}$. The assumption that $\mu_1 \geq 0$ yields the, fundamental, positivity preserving property \eqref{eq:Pipositive}. In addition, the symmetry properties (vanishing moments) of the function $\mu_{1,\vero'}$ imply that this operator preserves linears locally, therefore \eqref{eq:Pisuperapprox} holds.

We now define a positivity preserving interpolant $\Pi_{\T_\Y} : W^1_1(\calC_\Y) \to \polV(\T_\Y)$. Let $w \in W^1_1(\C_{\Y})$ and $(\vero',\vero'')$ be a node of $\T_\Y$, then
\begin{equation}
\label{eq:defofPi}
  \Pi_{\T_\Y} w(\vero',\vero'') = \begin{dcases}
                          \calL_{\T_\Y} w(\vero',\vero'') & \vero' > 0, \\
                          \calR_{\T_\Omega} w(\vero',0)  & \vero'' = 0.
                        \end{dcases}
\end{equation}

The approximation properties \eqref{eq:Piapprox} of the operator $\Pi_{\T_{\Y}}$ are as follows.

\begin{theorem}[interpolation estimate]
\label{TH:Pi_approx}
Let $s>3/8$. If $w \in \HL(y^{\alpha},\C_{\Y}) \cap H^2(y^{\beta},\C_{\Y})$ is $\alpha$-harmonic and $\tr w  \in \polH^{1+s}(\Omega)$, then 
\begin{equation}
\label{eq:Pi_estimate}
 \| \nabla( w - \Pi_{\T_{\Y}} w) \|_{L^2(y^{\alpha},\C_{\Y})} \lesssim ( \# \T_{\Y})^{-\frac{\theta}{(d+1)}}
\end{equation}
for $\theta < \theta_0 = \min\{1,\frac{8s-3}{4s} \}$. The hidden constant blows up as $\theta \uparrow \theta_0$, is independent of $\# \T_{\Y}$ and depends on $w$ through $\| w_{yy} \|_{L^2(y^{\beta},\C_{\Y})}$, $\| \nabla \nabla_{x'} w \|_{L^2(y^{\alpha},\C_{\Y})}$, $\| \nabla w \|_{L^2(y^{\alpha},\C_{\Y})}$ and $\| \tr w \|_{\polH^{1+s}(\Omega)}$.
\end{theorem}
\begin{proof}
Consider
\begin{multline*}
  \|\GRAD ( w - \Pi_{\T_\Y} w) \|_{L^2(y^\alpha,\calC_\Y)}^2 = 
  \|\GRAD ( w - \Pi_{\T_\Y} w) \|_{L^2(y^\alpha,\Omega \times (0,y_1))}^2 
  \\ + 
  \|\GRAD ( w - \Pi_{\T_\Y} w) \|_{L^2(y^\alpha,\Omega \times (y_1,\Y))}^2, 
\end{multline*}
where $y_1$ is defined by \eqref{eq:gradedmesh}. Since over $\Omega \times (y_1,\Y)$ the operators $\mathcal{L}_{\T_{\Y}}$ and $\Pi_{\T_{\Y}}$ coincide, we invoke 
\cite[Theorem 5.4]{NOS} to arrive at
\begin{equation}
\label{eq:w-Lw}
 \|\GRAD ( w - \Pi_{\T_\Y} w) \|_{L^2(y^\alpha,\Omega \times (y_1,\Y))} \lesssim (\# \T_{\Y})^{-1/(n+1)}, 
\end{equation}
where the hidden constant depends on the function $w$ only through its $\| w_{yy} \|_{L^2(y^{\beta},\C_{\Y})}$ and $\| \nabla \nabla_{x'} w \|_{L^2(y^{\alpha},\C_{\Y})}$ norms. We then need to estimate the remaining term $\|\GRAD ( w - \Pi_{\T_\Y} w) \|_{L^2(y^\alpha,\Omega \times (0,y_1))}$. To do this, we proceed as follows: 
\begin{multline*}
\|\GRAD ( w - \Pi_{\T_\Y} w) \|_{L^2(y^\alpha,\Omega \times (0,y_1))} \leq  
\|\GRAD ( w - \mathcal{L}_{\T_\Y} w) \|_{L^2(y^\alpha,\Omega \times (0,y_1))}
\\
+ \|\GRAD ( \mathcal{L}_{\T_\Y} w - \Pi_{\T_\Y} w) \|_{L^2(y^\alpha,\Omega \times (0,y_1))},
\end{multline*}
where $\mathcal{L}_{\T_\Y}$ is defined as in \eqref{eq:L}.  The first term of the expression above is controlled by the right hand side of \eqref{eq:w-Lw} by invoking, again, \cite[Theorem 5.4]{NOS}. To estimate the term
$ \mathcal{L}_{\T_\Y} w - \Pi_{\T_\Y} w$ over the first layer $\Omega \times (0,y_1)$, we use the definitions of $\mathcal{L}_{\T_\Y}$ and $\Pi_{\T_\Y}$ given by \eqref{eq:L} and \eqref{eq:defofPi}, respectively and exploit the fact that, for every node $\vero' \in \T_{\Omega}$, we have $\mathcal{L}_{\T_\Y} w (\vero',y_1) = \Pi_{\T_\Y} w (\vero',y_1)$, to write
\begin{multline*}
 \|\GRAD ( \mathcal{L}_{\T_\Y} w - \Pi_{\T_\Y} w) \|^2_{L^2(y^\alpha,\Omega \times (0,y_1))} = 
   \sum_{K \in \T_{\Omega}} \int_{0}^{y_1} y^\alpha \int_{K} |\GRAD ( \mathcal{L}_{\T_\Y} w - \Pi_{\T_\Y} w) |^2 \diff x' \diff y \\
 =  \sum_{K \in \T_{\Omega}}  \int_{0}^{y_1} y^\alpha \int_{K} 
   \left| \sum_{\vero' \in \N_{\Omega} } \left( \mathcal{L}_{\T_\Y} w(\vero',0) - \calR_{\T_\Omega} w(\vero',0)  \right) \GRAD(\Lambda_{\vero'}(x') \Lambda_0(y)) \right|^2\diff x' \diff y,
\end{multline*}
where $\calR_{\T_\Omega}$ is defined by \eqref{eq:CH_operator}, $\Lambda_{\vero'}$ denotes the basis function associated with the node $\vero'$ in the discretization $\T_{\Omega}$ and $\Lambda_0(y)$ the one that corresponds to the node $y_0=0$ in the discretization $\mathcal{I}_{\Y}$ defined by the mesh points \eqref{eq:gradedmesh}.

Using the well known finite intersection property of the supports of the basis functions $\Lambda_{\vero'}$ we see that, to conclude, it suffices to estimate, for $i=1,\ldots,d+1$,
\[
  \calE_i : = \int_{0}^{y_1} y^\alpha \int_{K} 
   \left| \left( \mathcal{L}_{\T_\Y} w(\vero',0) - \calR_{\T_\Omega} w(\vero',0)  \right) \partial_{x^i}(\Lambda_{\vero'}(x') \Lambda_0(y)) \right|^2\diff x' \diff y.
\]
To achieve this we notice, first of all, that if $i=1,\ldots,d$ we have
\begin{equation}
\label{eq:calEx}
  \calE_i \leq h_{\vero'}^{-2} |K| \left( \int_0^{y_1} y^\alpha \diff y \right) |\mathcal{L}_{\T_\Y} w(\vero',0) - \calR_{\T_\Omega} w(\vero',0)|^2,
\end{equation}
while, if $i=d+1$,
\begin{equation}
\label{eq:calEy}
  \calE_{d+1} \leq y_1^{-2} |K| \left( \int_0^{y_1} y^\alpha \diff y \right) |\mathcal{L}_{\T_\Y} w(\vero',0) - \calR_{\T_\Omega} w(\vero',0)|^2.
\end{equation}
We must uniformly bound the difference between $\calL_{\T_\Y} w $ and $\calR_{\T_\Omega} w$ over $\N_\Omega$.

Since $\int \mu_2(y) \diff y = 1$ and $\mu_1$ has vanishing moments, we have that, for \mae $z' \in \Omega$,
\begin{align*}
  \calR_{\T_\Omega} w (\vero',0) &= \int_{S_{\vero'}'} \mu_{1,\vero'}(x') w(x',0) \diff x' \\ &=
  \int_{S_{(\vero',0)}} \mu_{(\vero',0)}(x',y) \left[ w(x',0) + \GRAD_{x'}w(z',y)(\vero' - x') \right] \diff x' \diff y.
\end{align*}
Using this we see that, for any $\vero' \in \N_\Omega$, we have
\begin{multline*}
  (\calL_{\T_\Y} - \calR_{\T_\Omega} ) w(\vero',0) = 
  \int_{S_{(\vero',0)}} \mu_{(\vero',0)} \left[  w(x',y) - w(x',0) \right] \diff x' \diff y \\
  + \int_{S_{(\vero',0)}} \mu_{(\vero',0)} \GRAD_{x'} \left[ w(x',y) - w(z',y) \right](\vero' - x') \diff x' \diff y \\
  + \int_{S_{(\vero',0)}} \mu_{(\vero',0)}  \partial_y w(x',y) (0 - y ) \diff x' \diff y
  = \textrm{I} + \textrm{II} + \textrm{III}.
\end{multline*}
We now proceed to bound each one of these terms separately.

\begin{enumerate}[$\bullet$]

  \item \emph{Bound on} \textrm{III}: Using the scaling properties of the function $\mu_{(\vero',0)}$ and the Cauchy-Schwarz inequality we obtain
  \begin{align*}
    | \textrm{III} | \leq \frac{1}{h_{\vero'}^d} \left( \int_{S_{(\vero',0)}} y^{-\alpha} \diff x' \diff y \right)^{1/2} \| \partial_y w \|_{L^2(y^\alpha, S_{(\vero',0)})}.
  \end{align*}
  Since $w$ is $\alpha$-harmonic and \eqref{eq:gradedmesh} dictates the grading of the mesh $\T_{\Y}$ on the extended dimension,
  we have that $\| \partial_y w \|_{L^2(y^\alpha, S_{(\vero',0)})} \lesssim M^{-1}$, where $M^d \approx \# \T_{\Omega}$; see \cite[\S 5.2]{NOS}. This yields
  \begin{equation}
  \label{eq:III}
    | \textrm{III} | \lesssim \frac1{h_{\vero'}^d} \left( \int_{S_{(\vero',0)}} y^{-\alpha} \diff x' \diff y \right)^{1/2} M^{-1} \lesssim \frac1{h_{\vero'}^{d/2}} \left( \int_0^{y_1} y^{-\alpha} \diff y \right)^{1/2} M^{-1}.
  \end{equation}

  \item \emph{Bound on} \textrm{I}: A mean value result allows us to write
  \[
    |\textrm{I}| \leq \int_{S_{(\vero',0)}} \mu_{(\vero',0)} | \partial_y w(x',\eta(y))| y \diff x' \diff y.
  \]
  After this, we proceed as in the bound for \textrm{III} to conclude
  \begin{equation}
  \label{eq:I}
    | \textrm{I} | \lesssim \frac1{h_{\vero'}^d} \left( \int_{S_{(\vero',0)}} y^{-\alpha} \diff x' \diff y \right)^{1/2} M^{-1} \lesssim \frac1{h_{\vero'}^{d/2}} \left( \int_0^{y_1} y^{-\alpha} \diff y \right)^{1/2} M^{-1}.
  \end{equation}
  
  \item \emph{Bound on} \textrm{II}: We use that $w$ is $\alpha$-harmonic and a mean value theorem to get
  \begin{align*}
    | \textrm{II} | &\leq \frac1{h_{\vero'}^{d-2} y_1} \int_{S_{(\vero',0)}} | \GRAD_{x'}^2 w(\xi(x',z),y) | \diff x' \diff y \\
    &\leq \frac1{h_{\vero'}^{d-2} y_1} \left( \int_{S_{(\vero',0)}} y^{-\alpha} \diff x' \diff y \right)^{1/2} \| \GRAD_{x'}^2 w \|_{L^2(y^\alpha,S_{(\vero',0)})}.
  \end{align*}
  Using the local regularity results of \cite[Theorem 2.9]{NOS}, we derive
  \[
    \| \GRAD_{x'}^2 w \|_{L^2(y^\alpha,\Omega\times(0,y_1))}^2 \lesssim y_1 \| \tr w \|_{\polH^{1+s}(\Omega)}^2.
  \]
  This allows us to obtain
  \begin{equation}
  \label{eq:II}
   | \textrm{II} | \lesssim \frac1{h_{\vero'}^{d-2} y_1^{1/2} } \left( \int_{S_{(\vero',0)}} y^{-\alpha} \diff x' \diff y \right)^{1/2} \lesssim \frac1{h_{\vero'}^{d/2-2} y_1^{1/2}} \left( \int_0^{y_1} y^{-\alpha} \diff y \right)^{1/2}.
  \end{equation}
\end{enumerate}

To conclude we insert \eqref{eq:III}--\eqref{eq:II} into \eqref{eq:calEx} and \eqref{eq:calEy}. If $i=1,\ldots, d$ we have
\begin{align*}
  \calE_i &\lesssim h_{\vero'}^{d-2} \left( \int_0^{y_1} y^\alpha \diff y \right) \left[ \textrm{I}^2 + \textrm{II}^2 + \textrm{III}^2 \right] \\
  &\lesssim h_{\vero'}^{d-2} \left( \int_0^{y_1} y^\alpha \diff y \right) \left(\int_0^{y_1} y^{-\alpha} \diff y \right)
  \left[ \frac{M^{-2}}{h_{\vero'}^d}   + \frac{1}{h_{\vero'}^{d-4} y_1 } \right]
  = \calE_{i,1} + \calE_{i,2}.
\end{align*}
Using the fact that $y^\alpha \in A_2(\Real^{d+1})$
\[
  \calE_{i,1} = \left( \frac1{y_1}\int_0^{y_1} y^\alpha \diff y \right) \left(\frac1{y_1}\int_0^{y_1} y^{-\alpha} \diff y \right) \frac{y_1^2}{h_{\vero'}^2}  M^{-2}\lesssim ( \# \T_{\Y})^{-\frac{2}{(n+1)}},
\]
where we have used that $M \approx (\# \T_{\Omega})^{1/d} \approx (\# \T_{\Y})^{1/(d+1)}$ together with the fact that $y_1/h_{\vero'}$ is uniformly bounded. We now proceed to estimate the term $\calE_{i,2}$: 
\[
 \calE_{i,2} = \left( \frac1{y_1}\int_0^{y_1} y^\alpha \diff y \right) \left(\frac1{y_1}\int_0^{y_1} y^{-\alpha} \diff y \right) h_{\vero'}^2 y_1.
\]
Since $y^{\alpha} \in A_2(\R^{n+1})$ and $h_{\vero'} \approx M^{-2}$ we have $\calE_{i,2} \lesssim M^{-2}$. These yield $\calE_i \lesssim M^{-2}$.

Finally, we focus in the case $i=d+1$:
\[
 \calE_{d+1} \lesssim  \frac{h_{\vero'}^{d}}{y_1^2} \left( \int_0^{y_1} y^\alpha \diff y \right) \left[ \textrm{I}^2 + \textrm{II}^2 + \textrm{III}^2 \right]
 \lesssim M^{-2} + y_1^{-1} h_{\vero'}^4,
\]
where we have used, again, that $y^{\alpha} \in A_2(\R^{n+1})$. Since the grading parameter $\gamma$ satisfies $\gamma > 3/(2s)$, we have that, for $s>3/8$, $\theta_0 = \min\{1,(8s-3)/(4s)\}>0$. Therefore, we obtain that, for $\theta < \theta_0$, we can bound $\calE_{d+1} \lesssim M^{-\theta} \approx (\# \T_{\Y})^{-\theta/(d+1)}$. Collecting the derived estimates for $\calE_{i}$, with $i=1,\cdots,d$ and $\calE_{d+1}$ gives us \eqref{eq:Pi_estimate}.
\end{proof}

We now show the stability of the operator $\Pi_{\T_{\Y}}$ \eqref{eq:Pistable}.

\begin{corollary}[stability]
In the setting of Theorem~\ref{TH:Pi_approx}, we have 
\[
 \|  \nabla \Pi_{\T_{\Y}} w\|_{L^2(y^{\alpha},\C_{\Y})} \lesssim \| w_{yy} \|_{ L^2( y^{\beta},\C_{\Y} )} + \| \nabla \nabla_{x'} w \|_{L^2(y^{\alpha},\C_{\Y})}
 + \| \tr w\|_{\polH^{1+s}(\Omega)},
\]
\end{corollary}
where the hidden constant does not depend on $\T_\Y$.
\begin{proof}
The result follows as a simple application of Theorem~\ref{TH:Pi_approx}. In fact,
\begin{align*}
\|  \nabla \Pi_{\T_{\Y}} w\|_{L^2(y^{\alpha},\C_{\Y})}  & \leq  
\|  \nabla ( w - \Pi_{\T_{\Y}} w) \|_{L^2(y^{\alpha},\C_{\Y})} +
\|  \nabla w \|_{L^2(y^{\alpha},\C_{\Y})}
\end{align*}
and \eqref{eq:Pi_estimate} yield the desired estimate.
\end{proof}

To conclude let us use this operator to obtain error estimates.

\begin{corollary}[error estimate for $\ve$]
\label{col:errestU}
Assume that $s>3/8$ and set
\[
  \vartheta_0 = -1, \ s \in \left[ \frac34, 1 \right), \qquad \vartheta = \frac{8s-3}{4s}, \ s \in \left( \frac38, \frac34 \right). 
\]
In the setting of Theorem~\ref{thm:optapriori}, if $\dt \approx (\# \T_{\Y})^{\frac{-1}{d+1}}$, then
\begin{equation}
\label{eq:final_estimate}
E(\ve,V^\dt_{\T_\Y})
\lesssim |\log \# \T_{\Y}|^s (\# \T_{\Y})^{-\frac{\vartheta}{d+1}}
\end{equation}
for $\vartheta < \vartheta_0$. The hidden constant blows up as $\vartheta \uparrow \vartheta_0$, is independent of $\# \T_{\Y}$ and depends only on the problem data.
\end{corollary}
\begin{proof}
The choice of $\tau$ and, depending on the value of $s$, a comparison of the terms $(\# \T_{\Y})^{-\theta/(d+1)}$
and $(\# \T_{\Y})^{-(1+s)/(d+1)} \tau^{-1/2}$ on the right-hand side of the estimate of Theorem \ref{thm:optapriori} yields the result.
\end{proof}

\begin{corollary}[error estimate for $\ue$]
\label{col:errestu}
Assume that $s>3/8$. In the setting of Theorem~\ref{thm:optapriori} and Corollary~\ref{col:errestU} we have
\[
\calE(\ue,U^\dt_{\T_\Omega})
\lesssim |\log \# \T_{\Y}|^s (\# \T_{\Y})^{-\frac{\vartheta}{d+1}}
\]
for $\vartheta < \vartheta_0$. The hidden constant blows up as $\theta \uparrow \theta_0$, is independent of $\# \T_{\Y}$ and depends only on the problem data.
\end{corollary}
\begin{proof}
The desired estimate follows from Corollary \ref{col:semiimplII} and the estimate \eqref{eq:final_estimate}.
\end{proof}

\bibliographystyle{plain}
\bibliography{biblio}

\def\cprime{$'$} \def\cprime{$'$} \def\cprime{$'$} \def\cprime{$'$}
  \def\cprime{$'$} \def\cprime{$'$}
\begin{thebibliography}{10}

\bibitem{Adams}
R.A. Adams.
\newblock {\em Sobolev spaces}.
\newblock Academic Press, 1975.

\bibitem{AllenPetrosyan}
M.~Allen, E.~Lindgren, and A.~Petrosyan.
\newblock The two-phase fractional obstacle problem.
\newblock {\em SIAM Journal on Mathematical Analysis}, 47(3):1879--1905, 2015.

\bibitem{atanackovic2014fractional}
T.M. Atanackovic, S.~Pilipovic, B.~Stankovic, and D.~Zorica.
\newblock {\em Fractional Calculus with Applications in Mechanics: Vibrations
  and Diffusion Processes}.
\newblock John Wiley \& Sons, 2014.

\bibitem{BS}
M.{\v{S}}. Birman and M.Z. Solomjak.
\newblock {\em Spektralnaya teoriya samosopryazhennykh operatorov v gilbertovom
  prostranstve}.
\newblock Leningrad. Univ., Leningrad, 1980.

\bibitem{Brezis}
H.~Br{\'e}zis.
\newblock {\em Op\'erateurs maximaux monotones et semi-groupes de contractions
  dans les espaces de {H}ilbert}.
\newblock North-Holland, 1973.

\bibitem{bio}
A.~Bueno-Orovio, D.~Kay, V.~Grau, B.~Rodriguez, and K.~Burrage.
\newblock Fractional diffusion models of cardiac electrical propagation: role
  of structural heterogeneity in dispersion of repolarization.
\newblock {\em J. R. Soc. Interface}, 11(97), 2014.

\bibitem{MR3100955}
L.~Caffarelli and A.~Figalli.
\newblock Regularity of solutions to the parabolic fractional obstacle problem.
\newblock {\em J. Reine Angew. Math.}, 680:191--233, 2013.

\bibitem{CS:07}
L.~Caffarelli and L.~Silvestre.
\newblock An extension problem related to the fractional {L}aplacian.
\newblock {\em Comm. Part. Diff. Eqs.}, 32(7-9):1245--1260, 2007.

\bibitem{wow}
W.~Chen.
\newblock A speculative study of $2/3$-order fractional laplacian modeling of
  turbulence: Some thoughts and conjectures.
\newblock {\em Chaos}, 16(2):1--11, 2006.

\bibitem{MR1742264}
Z.~Chen and R.H. Nochetto.
\newblock Residual type a posteriori error estimates for elliptic obstacle
  problems.
\newblock {\em Numer. Math.}, 84(4):527--548, 2000.

\bibitem{CiarletBook}
P.G. Ciarlet.
\newblock {\em The finite element method for elliptic problems}, volume~40 of
  {\em Classics in Applied Mathematics}.
\newblock SIAM, Philadelphia, PA, 2002.

\bibitem{Javier}
J.~Duoandikoetxea.
\newblock {\em Fourier analysis}.
\newblock American Mathematical Society, Providence, RI, 2001.

\bibitem{Guermond-Ern}
A.~Ern and J.-L. Guermond.
\newblock {\em Theory and practice of finite elements}.
\newblock Springer, New York, 2004.

\bibitem{GH:14}
P.~Gatto and J.~Hesthaven.
\newblock Numerical approximation of the fractional {L}aplacian via hp-finite
  elements, with an application to image denoising.
\newblock {\em J. Sci. Comp.}, pages 1--22, 2014.
\newblock DOI:10.1007/s10915-014-9959-1.

\bibitem{HB:10}
Y.~Ha and F.~Bobaru.
\newblock Studies of dynamic crack propagation and crack branching with
  peridynamics.
\newblock {\em Int. J. Fracture}, 162(1-2):229--244, 2010.

\bibitem{HKM}
J.~Heinonen, T.~Kilpel{\"a}inen, and O.~Martio.
\newblock {\em Nonlinear potential theory of degenerate elliptic equations}.
\newblock Oxford University Press, New York, 1993.

\bibitem{ICH}
R.~Ishizuka, S.-H. Chong, and F.~Hirata.
\newblock An integral equation theory for inhomogeneous molecular fluids: The
  reference interaction site model approach.
\newblock {\em J. Chem. Phys}, 128(3), 2008.

\bibitem{KS}
D.~Kinderlehrer and G.~Stampacchia.
\newblock {\em An introduction to variational inequalities and their
  applications}, volume~88 of {\em Pure and Applied Mathematics}.
\newblock Academic Press, 1980.

\bibitem{MR0057472}
P.P. Korovkin.
\newblock On convergence of linear positive operators in the space of
  continuous functions.
\newblock {\em Doklady Akad. Nauk SSSR (N.S.)}, 90:961--964, 1953.

\bibitem{MR0089939}
P.P. Korovkin.
\newblock On the order of the approximation of functions by linear positive
  operators.
\newblock {\em Dokl. Akad. Nauk SSSR (N.S.)}, 114:1158--1161, 1957.

\bibitem{Kufner80}
A.~Kufner.
\newblock {\em Weighted {S}obolev spaces}.
\newblock Teubner, Leipzig, 1980.

\bibitem{MR2064019}
S.~Z. Levendorski{\u\i}.
\newblock Pricing of the {A}merican put under {L}\'evy processes.
\newblock {\em Int. J. Theor. Appl. Finance}, 7(3):303--335, 2004.

\bibitem{Lions}
J.-L. Lions and E.~Magenes.
\newblock {\em Non-homogeneous boundary value problems and applications. {V}ol.
  {I}}.
\newblock Springer-Verlag, New York, 1972.

\bibitem{NOS4}
R.H. Nochetto, E.~Ot\'arola, and A.J. Salgado.
\newblock Convergence rates for the obstacle problem: classical, thin and
  fractional.
\newblock Phil. Trans. R. Soc. A (accepted), 2014.

\bibitem{NOS3}
R.H. Nochetto, E.~Ot\'arola, and A.J. Salgado.
\newblock A {PDE} approach to space-time fractional parabolic problems.
\newblock arXiv:1404.0068, 2014.

\bibitem{NOS}
R.H. Nochetto, E.~Ot{\'a}rola, and A.J. Salgado.
\newblock A {PDE} {A}pproach to {F}ractional {D}iffusion in {G}eneral
  {D}omains: {A} {P}riori {E}rror {A}nalysis.
\newblock {\em Found. Comput. Math.}, 15(3):733--791, 2015.
\newblock DOI:10.1007/s10208-014-9208-x.

\bibitem{NOS2}
R.H. Nochetto, E.~Ot{\'a}rola, and A.J. Salgado.
\newblock Piecewise polynomial interpolation in {M}uckenhoupt weighted
  {S}obolev spaces and applications.
\newblock {\em Numer. Math.}, pages 1--46, 2015.
\newblock DOI:http://dx.doi.org/10.1007/s00211-015-0709-6.

\bibitem{NSV}
R.H. Nochetto, G.~Savar{\'e}, and C.~Verdi.
\newblock A posteriori error estimates for variable time-step discretizations
  of nonlinear evolution equations.
\newblock {\em Comm. Pure Appl. Math.}, 53(5):525--589, 2000.

\bibitem{NTZ:10}
R.H. Nochetto, T.~von Petersdorff, and C.-S. Zhang.
\newblock A posteriori error analysis for a class of integral equations and
  variational inequalities.
\newblock {\em Numer. Math.}, 116(3):519--552, 2010.

\bibitem{MR1933037}
R.H. Nochetto and L.B. Wahlbin.
\newblock Positivity preserving finite element approximation.
\newblock {\em Math. Comp.}, 71(240):1405--1419, 2002.

\bibitem{MR1424787}
H.~Pham.
\newblock Optimal stopping, free boundary, and {A}merican option in a
  jump-diffusion model.
\newblock {\em Appl. Math. Optim.}, 35(2):145--164, 1997.

\bibitem{Tartar}
L.~Tartar.
\newblock {\em An introduction to {S}obolev spaces and interpolation spaces}.
\newblock Springer, Berlin, 2007.

\bibitem{Turesson}
B.O. Turesson.
\newblock {\em Nonlinear potential theory and weighted {S}obolev spaces}.
\newblock Springer, Berlin, 2000.

\end{thebibliography}

\end{document}